\documentclass{article}
\usepackage[utf8]{inputenc}
\usepackage[left=3cm,right=3cm,top=3cm,bottom=3cm]{geometry}
\usepackage{graphicx}
\usepackage{enumitem}
\usepackage{color}
\usepackage{amsmath,amsfonts,amssymb,amsthm,epsfig,epstopdf,titling,url,array}
\usepackage{mathrsfs}
\usepackage{caption}
\usepackage{subcaption}
\usepackage{tikz-cd}
\usepackage{MnSymbol}

\DeclareMathAlphabet{\mathpzc}{OT1}{pzc}{m}{it}

\DeclareMathOperator{\conv}{Conv}
\DeclareMathOperator{\Ham}{Ham}

\DeclareMathOperator{\osc}{Osc}

\newcommand{\Mane}{Mañé }

\def\XXint#1#2#3{{\setbox0=\hbox{$#1{#2#3}{\int}$ }
\vcenter{\hbox{$#2#3$ }}\kern-.6\wd0}}

\newtheorem{theo}{Theorem}[section]
\newtheorem{prop}[theo]{Proposition}
\newtheorem{lem}[theo]{Lemma}
\newtheorem{cor}[theo]{Corollary}

\theoremstyle{definition}
\newtheorem{defi}[theo]{Definition}

\newtheorem{rem}[theo]{Remark}

\newtheorem{quest}{Question}

\theoremstyle{remark}

\numberwithin{equation}{section}

\title{A Multidimensional Birkhoff Theorem for Recurrent Lagrangian Submanifolds by a Tonelli Hamiltonian}
\author{Skander Charfi}
\date{}

\begin{document}

\maketitle

\begin{abstract}
	Consider a closed manifold $M$ and a time-periodic Tonelli Hamiltonian $H : \mathbb{R}/\mathbb{Z} \times T^*M \to \mathbb{R}$ with flow $\phi_H$. Let $\mathcal{L} \subset T^*M$ be a Lagrangian submanifold Hamiltonianly isotopic to the zero section. We prove that if $\phi_H^n(\mathcal{L})$ admits convergent subsequences in both positive and negative times, in the Hausdorff topology and with control on the Liouville primitives, to two Lagrangian submanifolds, then $\mathcal{L}$ is a graph over the zero section $0_{T^*M}$ of $T^*M$. Furthermore, we show that $\mathcal{L}$ is recurrent in both positive and negative times for the same type of convergence.
\end{abstract}

\tableofcontents

\newpage

\addcontentsline{toc}{section}{Introduction}
\section*{Introduction}

\subsection{Overview}

A theorem attributed to G. D. Birkhoff \cite{MR1555175} establishes that any non contractible invariant curve, which is preserved under an exact twist map of the annulus, is a Lipschitz graph over the circle. This theorem has inspired numerous efforts to establish various modern proofs. We can refer for example to \cite{MR0728564,MR1476996,MR1659220}.\\

Since then, attempts were made to extend the result to higher dimensions. Under certain assumptions, the authors have managed to demonstrate that for a convex Hamiltonian of a cotangent bundle or a multidimensional positive twist map, an invariant exact Lagrangian submanifold is a graph.\\

	For results on the multidimensional torus $\mathbb{T}^n$ with various conditions on the invariant Lagrangian, we can cite the works of Herman \cite{MR1067380}, Bialy and Polterovich \cite{MR1157317}, and Carneiro and Ruggiero \cite{MR4597700}. On general manifolds $M$, Arnaud \cite{MR2738994} proved that if a Lagrangian submanifold of $T^*M$, which is Hamiltonianly isotopic to the zero section, is fixed by an autonomous Tonelli Hamiltonian map, then it is a Lipschitz graph over the base manifold $M$. Later, Bernard and dos Santos \cite{MR2860422} extended this result to the case of Lipschitz Lagrangian submanifolds, and Amorim, Oh, and dos Santos \cite{MR3860396} dropped the Hamiltonian isotopy condition, proving the theorem for exact-Lipschitz Lagrangian submanifolds using generalized graph selectors based on Floer theory.\\

In the non-autonomous setting, Arnaud and Venturelli \cite{MR3674224} showed that the result of \cite{MR2738994} still holds. It is established that the periodic Lagrangian submanifold is the graph of some $du$, where $u$ is a weak-KAM solution of the Hamilton-Jacobi equation. This finding has been the starting point of this paper, suggesting a correspondence between larger sets of Lagrangian submanifolds (Hamiltonianly isotopic the the zero section) and solutions of the Hamilton-Jacobi equation.\\  

Periodic Lagrangian submanifolds are replaced by Lagrangian submanifolds $\mathcal{L}$ which images under the action of a Tonelli Hamiltonian flow $\phi_H$ have convergent subsequences in both positive and negative times, for a type of convergence called reduced complexity convergence (see Definition \ref{ReducedComplexity}). In other words, we require that two subsequences of $\phi_H^n(\mathcal{L})$ and of $\phi_H^{-n}(\mathcal{L})$ converge in the Hausdorff topology to another Lagrangian submanifold, with a control on the Liouville primitive to prevent winding phenomena.\\

In addition, the author noted in \cite{Representation} that bounded global viscosity solutions of the Hamilton-Jacobi equation are recurrent and share the same properties as weak-KAM solutions, particularly those used to prove the Birkhoff theorem in \cite{MR3674224}.\\

As a result, it is established in this paper that Lagrangian submanifolds that converge with reduced complexity in both positive and negative times to some limit Lagrangian subamnifolds are graphs of $du$, where $u$ is a bounded viscosity solution of the Hamilton-Jacobi equation. Consequently, it is shown that these submanifolds are recurrent under the action of the Hamiltonian flow, meaning that $\phi_H^n(\mathcal{L})$ has a subsequence that converges, with reduced complexity, in both positive and negative times, to the initial Lagrangian $\mathcal{L}$.\\

One difficulty is to define the appropriate topology for the recurrence of Lagrangian submanifolds in order to obtain the desired result. If the chosen topology is too lenient, it might permit the Lagrangians to wind around their limit, leading to a failure of the graph property. It will be shown that controlling the stretching of the submanifolds will suffice. See Definition \ref{ReducedComplexity} and Theorem \ref{MainTheorem}.\\
 
Another difficulty is the construction of concrete examples of recurrent Lagrangian submanifolds under Tonelli Hamiltonian maps. This complicates the construction of counterexamples for loose topologies where recurrence does not necessarily imply a Birkhoff result. However, to demonstrate that the framework of this paper is non-empty, the author constructed an example of a recurrent, non-periodic $C^1$ Lagrangian submanifold that is the graph of a recurrent viscosity solution \cite{Recurrent}.\\

To prove the main result, we rely heavily on the weak KAM theory, which was developed by A. Fathi in the 1990s \cite{MR1451248,fathi2008weak}. However, it is important to note that we deliberately refrained from relying on the entire theory, including concepts like the Aubry sets and weak KAM solutions. This deliberate choice ensures that our article stands as a self-contained work. An exception will be made for Corollary \ref{MainCor2}, the proof of which requires the properties of the Lax-Oleinik operators briefly discussed in Appendix \ref{AppendixLO}.\\

Additionally, a crucial element of our proof involves the use of a graph selector. These allow us to choose a pseudograph (a type of discontinuous exact Lagrangian graph) within the initial Lagrangian submanifold. The concept of graph selectors was first introduced by M. Chaperon \cite{MR1094198,MR2025275}. For our purposes, we will adopt the construction given by C. Viterbo in \cite{MR1604386} relying on spectral invariants defined from generating functions (see \cite{MR1157321} or \cite{humil2008} Chapter 1). For the sake of completeness, we have chosen to reprove the vast majority of the properties of these invariants by adapting them to the framework of this paper.

\subsection{Notations and Main Result}

Fix a closed manifold $M$ of dimension $d$ endowed with a Riemannian metric $d$ and its relative norm $\Vert .\Vert $ on the cotangent bundle. The cotangent bundle $T^*M$ has a natural exact symplectic structure $(T^*M, \omega = -d\lambda)$ defined as follows. Let $\pi_M : T^*M \to M$ be the natural projection and denote by $(q,p)=(q_1,..,q_d,p_1,..,p_d)$ the coordinates in $T^*M$ where $q = (q_1,..,q_d)$ are local coordinates of $M$ and $p = (p_1, .., p_d)$ are fiberwise coordinates with respect to the cotangent vectors $dq_1,..,dq_d$. The \textit{Liouville form} $\lambda$ is defined by $\lambda(q,p) = p \circ d\pi_M = p.dq$.

A \textit{Lagrangian submanifold} $\mathcal{L}$ of $T^*M$ is a submanifold such that $\dim \mathcal{L} =d$ and $\omega_{|T \mathcal{L}} =0$. The Lagrangian submanifold $\mathcal{L}$ is \textit{exact} if $\lambda_{|T \mathcal{L}}$ is an exact form i.e there exist a \textit{Liouville primitive} $h : \mathcal{L} \to \mathbb{R}$ such that $\lambda_{|\mathcal{L}} = dh$. We define the \textit{oscillation} of $h$ as
\begin{equation}
	\osc(h) = \max h - \min h
\end{equation}

The \textit{Hausdorff distance} $d_H$ on the set of compact Lagrangian submanifolds of $T^*M$ is defined for all such $\mathcal{L}$ and $\mathcal{L}'$ as
\begin{equation}
	d_H(\mathcal{L}, \mathcal{L}') = \max \Big\{ \sup_{x' \in \mathcal{L}'} d(x', \mathcal{L}) \; , \; \sup_{x \in \mathcal{L}} d(x, \mathcal{L}') \Big\} 
\end{equation}

Set $\mathbb{T}^1 = \mathbb{R}/ \mathbb{Z}$ and denote by $t$ the time coordinate in $\mathbb{T}^1$. A $C^2$ time-periodic \textit{Hamiltonian} is a map $H : \mathbb{T}^1 \times T^*M \to \mathbb{R}$. Given $H$, the \textit{Hamiltonian vector field} $X_H$ is defined by the equation $\iota_{X^t_H} \omega = \omega( X^t_H, \cdot ) = dH_t$ where $H_t = H(t,\cdot,\cdot)$ and the corresponding \textit{Hamiltonian flow} is denoted by $\phi_H^{s,t}$. We set $\phi_H^t := \phi_H^{0,t}$. The \textit{Hamiltonian group} $\Ham(T^*M, \omega)$ is the group of Hamiltonian maps i.e time one of Hamiltonian flows.

\begin{defi} \label{Tonelli}
	 A $C^2$ time-periodic Hamiltonian $H(t,q,p) : \mathbb{T}^1 \times T^*M \to \mathbb{R}$ is called \textit{Tonelli} if it satisfies the following classical hypotheses:
	\begin{itemize}
		\item (\textit{Strict convexity}) $\partial_{pp}H(t,q,p)>0$ for all $(t,q,p) \in \mathbb{T}^1 \times T^*M$.
		\item (\textit{Superlinearity}) $\frac{|H(t,q,p)|}{\Vert p\Vert } \to \infty$ as $\Vert p\Vert  \to \infty$ for each $(t,q) \in \mathbb{T}^1 \times M$.
		\item (\textit{Completeness}) The Hamiltonian vector field $X_H$ and hence its flow $\phi_H^{s,t}$ are complete in the sense that the flow curves are defined for all times $t \in \mathbb{R}$.
	\end{itemize}
\end{defi}

Tonelli Hamiltonians are the good setting to have a correspondence between Hamiltonian and Lagrangian dynamics. This is precisely the framework of Fathi's weak KAM theory. A flavour of this will be given in Section \ref{Proofsection} and Appendix \ref{AppendixLO}. For a more detailed exposition on the subject, we refer to \cite{fathi2008weak}.\\

We now fix a Lagrangian submanifold $\mathcal{L}$ that is \textit{Hamiltonianly isotopic} or \textit{$H$-isotopic} to the zero section $0_{T^*M}$, that is, there exists a Hamiltonian map $\varphi \in \Ham(T^*M,\omega)$ such that $\mathcal{L} = \varphi (0_{T^*M})$. For all time $t$ in $\mathbb{R}$, we set $\mathcal{L}_t = \phi_H^{t}(\mathcal{L})$ and $\mathcal{L}'_t = \varphi^{-1}(\mathcal{L}_t)$.

\begin{defi} \label{ReducedComplexity}
	Let $(\mathcal{L}_n)_{n \geq 0}$ and $\mathcal{L}$ be a exact Lagrangian submanifolds of $T^*M$ $H$-isotopic to the null section $0_{T^*M}$. We say that the sequence $(\mathcal{L}_n)_{n \geq 0}$ \textit{converges with reduced complexity} if
	\begin{enumerate}[label = \roman*.]
		\item $\lim\limits_n d_H(\mathcal{L}_n , \mathcal{L}) =0$.
		\item For a Hamiltonian map $\varphi$ such that $\mathcal{L} = \varphi ( 0_{T^*M})$,  if $l_n$ is a Liouville primitive on the Lagrangian submanifold $\varphi^{-1} ( \mathcal{L}_n)$, then $\lim\limits_n \osc (l_n) =0$.
	\end{enumerate}
\end{defi}

\begin{prop} \label{ReducedComplexityProp}
	The definition of reduced complexity convergence does not depend on the choice of the Hamiltonian map $\varphi$ such that $\mathcal{L} = \varphi(0_{T^*M})$.
\end{prop}

\begin{rem}
	\begin{enumerate}
		\item The idea of this convergence is to constraint the winding of the Lagrangian submanifolds around the limit points. More precisely, if $h_n$ and $h_\omega$ are Liouville primitives on $\mathcal{L}_n$ and $\mathcal{L}_\omega$, we aim to control the oscillation of $"h_n - h_\omega"$. However, the two primitives have different definition domains and we need to make a change of variable in order to reduce $h_\omega$ into a constant primitive $l_\omega$ of $\varphi_\omega^{-1}(\mathcal{L}_\omega) = 0_{T^*M}$.
		\item Comparing the definition to other types of convergence, we have
		\begin{enumerate}
			\item The $C^1$ convergence of $(\mathcal{L}_n)_n$ to $\mathcal{L}_\omega$ implies reduced complexity convergence.
			\item Reduced complexity convergence implies spectral convergence by the mean of the spectral distance $\gamma$ introduced by C.Viterbo in \cite{MR1157321}. More precisely, if $(\mathcal{L_n}_n)_{n\geq 0}$ converges with reduced complexity to $\mathcal{L}_\omega$, then $\lim_n \gamma(\mathcal{L}_n, \mathcal{L}_\omega) =0$. However the converse is false since spectral convergence allows non-boundedness of $\osc (l_n)$.
		\end{enumerate}				
	\end{enumerate}
\end{rem}

We state the main theorem of the paper.

\begin{theo}  \label{MainTheorem}
	Let $M$ be a closed manifold, $H : \mathbb{T}^1 \times T^*M \to \mathbb{R}$ be a Tonelli Hamiltonian with flow $\phi_H$ and let $\mathcal{L}$ be a Lagrangian submanifold of $T^*M$ which is $H$-isotopic to the zero section. For all time $t \in \mathbb{R}$, we set $\mathcal{L}_t = \phi_H^t(\mathcal{L})$. Suppose that there exist two Lagrangian submanifolds $\mathcal{L}_\omega$ and $\mathcal{L}_\alpha$ $H$-isotopic to the zero section, and two increasing sequences of integers $n_k$ and $m_k$ such that $(\mathcal{L}_{n_k})_{k \geq 0}$ and $(\mathcal{L}_{-m_k})_{k \geq 0}$ converge with reduced complexity to $\mathcal{L}_\omega$ and $\mathcal{L}_\alpha$ respectively. Then $\mathcal{L}$ and all its images $\mathcal{L}_t$ are $C^1$ graphs over the zero section $0_{T^*M}$ of $T^*M$.
\end{theo}

In particular, this Theorem applies for reccurent Lagrangian submanifolds with a control on the Liouville primitives.

\begin{cor} \label{MainCorRecu}
	Let $M$ be a closed manifold, $H : \mathbb{T}^1 \times T^*M \to \mathbb{R}$ be a Tonelli Hamiltonian with flow $\phi_H$ and let $\mathcal{L}$ be a Lagrangian submanifold of $T^*M$ which is $H$-isotopic to the zero section. If $\mathcal{L}$ is a positively and negatively time recurrent Lagrangian for the reduced complexity convergence, meaning that if there exist increasing sequences of integers $n_k$ and $m_k$ such that $(\mathcal{L}_{n_k})_k$ and $(\mathcal{L}_{-m_k})_k$ converge with reduced complexity to $\mathcal{L}$, then $\mathcal{L}$ and all its images $\mathcal{L}_t$ are $C^1$ graphs over the zero section $0_{T^*M}$ of $T^*M$. 
\end{cor}

An immediate consequence is the result by Marie-Claude Arnaud and Andrea Venturelli \cite{MR3674224} on periodic Lagrangian submanifolds under Tonelli Hamiltonian maps.

\begin{cor}
	Let $M$ be a closed manifold. If a Lagrangian submanifold $\mathcal{L}$ of $T^*M$ is periodic under the Hamiltonian map $\phi^1_H$ of a Tonelli Hamiltonian $H$, then $\mathcal{L}$ is a $C^1$ graph over the zero section $0_{T^*M}$ of $T^*M$. 
\end{cor}

It is possible to obtain more information about the scalar map $u : \mathbb{R} \times M \to \mathbb{R}$ for which $\mathcal{L}_t$ is the graph of $du(t, \cdot)$. Further study, focused on weak-KAM theory and viscosity solutions theory, reveals that $u(t, q)$ is a recurrent (viscosity) solution of the Hamilton-Jacobi equation
\begin{equation}
	\partial_tu + H(t,q,d_qu(t,q)) =\alpha_0
\end{equation}
where $\alpha_0$ is a well chosen constant (see Proposition \ref{ManeCritValueProp}). This is contained in the establishment of following.

\begin{cor} \label{MainCor2} 
	Under the assumptions of Theorem \ref{MainTheorem}, the Lagrangian submanifold $\mathcal{L}$ is a $\phi_H^1$-recurrent Lagrangian submanifold for the Hausdorff distance.
\end{cor}

In the autonomous case, A.Fathi \cite{MR1650261,Representation} established a convergence theorem which tells that recurrent (viscosity) solutions of the Hamilton-Jacobi equation are stationary, meaning that they are time-independent and thus satisfy the equation
\begin{equation}
	H(q,d_qu(t,q)) = \alpha_0
\end{equation}
As a consequence, we derive the following autonomous version of Corollary \ref{MainCor2}.

\begin{cor} \label{MainCor3} 
	Under the assumptions of Theorem \ref{MainTheorem} and if $H : T^*M \to \mathbb{R}$ is autonomous, the Lagrangian submanifold $\mathcal{L}$ is a $C^1$ graph over the zero section $0_{T^*M}$, and it is a $\phi_H$-invariant Lagrangian submanifold i.e for all time $t \in \mathbb{R}$, $\phi_H^t(\mathcal{L}) = \mathcal{L}$.
\end{cor} 

\begin{rem} \label{MainRem}
Many questions arise about possible generalizations:
\begin{enumerate}
	\item (Types of recurrence) The reduced complexity convergence may be considered too restrictive. The proof presented here works when the variation between the Liouville primitives on $\mathcal{L}_{-m_k}$ and $\mathcal{L}_{n_k}$ converges to zero as this provides calibration for limit curves that will allow the use of results from Fathi's weak-KAM theory. The conclusion we obtain is that $\mathcal{L}$ is the graph of $du$ where $u$ is a (viscosity) solution of the Hamilton-Jacobi equation.\\ 
	However, this reasoning fails when this variation between the Liouville primitives grows as it may happen if we only assume the Hausdorff convergence. In this case, we failed to make a working proof nor could we provide a counter-example. 
	\begin{quest} 
		Is a Birkhoff theorem still valid if we only assume a Hausdorff convergence, or more reasonably, spectral convergence of the sequences $\mathcal{L}_{-m_k}$ and $\mathcal{L}_{n_k}$ respectively to $\mathcal{L}_\alpha$ and $\mathcal{L}_\omega$. Here, spectral convergence stands for Viterbo's gamma distance $\gamma$ on Lagrangian submanifolds.
	\end{quest}
	
	\item (Negative times) It seems that for the case of viscosity solutions $u(t,x)$, negative time-recurrence implies positive-time recurrence (see Appendix \ref{AppendixLO}). Therefore, this is the case of the Lagrangian $\mathcal{L} = \mathcal{G}(d_qu(0, \cdot))$. And since the Lagrangian submanifolds considered in Theorem \ref{MainTheorem} turn out to be graphs of the differential of recurrent viscosity solutions, it feels natural to ask the question.
	\begin{quest}
		Is the Birkhoff theorem still true if we only assume a negative-time reduced complexity convergence? Alternatively, is it possible to construct a negative-time-recurrent Lagrangian submanifold $\mathcal{L}$ that is not recurrent in positive times ?
	\end{quest}
	Providing a counter-example turned out to be trickier than expected as examples of recurrent Lagrangian submanifolds are lacking to the literature.
\end{enumerate}
\end{rem}

\begin{proof}[Proof of Proposition \ref{ReducedComplexityProp}] 
	Let $\varphi$ and $\psi$ be two Hamiltonian maps such that $\mathcal{L} = \varphi(0_{T^*M}) = \psi(0_{T^*M})$. We set $k_n$ to be Liouville primitives on the Lagrangian submanifolds $\varphi^{-1} (\mathcal{L}_n)$. We get Liouville primitives $l_n$ on $\psi^{-1} (\mathcal{L}_n)$ using the following lemma.
	\begin{lem} \label{PrimitiveChange}
		Let $f$ be an exact symplectomorphism of $T^*M$ and let $g : T^*M \to \mathbb{R}$ be a scalar map such that $f^* \lambda - \lambda = dg$. Then, if $\mathcal{L}$ is an exact Lagrangian submanifold of $T^*M$ with Liouville primitive $k$, then $f^{-1}(\mathcal{L})$ is an exact Lagrangian submanifold of $T^*M$ with Liouville primitive $l=k\circ f -g$.
	\end{lem}	
	\begin{proof}
		Let $x$ and $y$ be two points of $f^{-1}(\mathcal{L})$ linked by a curve $\gamma$ in $f^{-1}(\mathcal{L})$ and let $x' = f(x)$ and $y' = f(y)$ linked by $\gamma' = f(\gamma)$ in $\mathcal{L}$. We will evaluate $\int_\gamma \lambda$.
		We have 
		\begin{align*}
			(f^{-1})^*\lambda = (f^{-1})^* ( f^* \lambda - dg) = \lambda +d(-g\circ f^{-1})
		\end{align*}
		We set $h := -g\circ f^{-1}$. Then we have
		\begin{align*}
			\int_\gamma \lambda &= \int_{\gamma'} (f^{-1})^*\lambda = \int_{\gamma'} \lambda +dh = \int_{\gamma'} dk +dh = \int_{\gamma'} d(k+h) = \int_\gamma d\big( (k+h) \circ f \big)
		\end{align*}
		where 
		\begin{align*}
			(k+h) \circ f  = k \circ f - g = l
		\end{align*}
		Hence, the Lagrangian submanifold $f^{-1}(\mathcal{L})$ is exact with Liouville primitive $l$.
	\end{proof}             
	
	We apply the lemma for $f = \varphi^{-1} \circ \psi$ so that $\psi^{-1}(\mathcal{L}_n) = f^{-1} ( \varphi^{-1}(\mathcal{L}_n))$ and we choose $l_n = k_n \circ f -g$ to be Liouville primitives on $\psi^{-1}(\mathcal{L}_n)$, where $f^*\lambda - \lambda = dg$. We need to prove that $\lim_n \osc (l_n) = 0$. We have
	\begin{equation*}
		\osc (l_n) = \osc (k_n \circ f -g) \leq \osc (k_n \circ f) + \osc (g_{|\psi^{-1}(\mathcal{L}_n)}) = \osc (k_n) + \osc (g_{|\psi^{-1}(\mathcal{L}_n)})
	\end{equation*}
	Since the reduced complexity convergence provides the first limit $\lim_n \osc (k_n) =0$, it suffices to show that $\lim_n \osc (g_{|\psi^{-1}(\mathcal{L}_n)})$. We know that $\mathcal{L} = \varphi (0_{T^*M}) = \psi( 0_{T^*M})$ so that $f(0_{T^*M}) = 0_{T^*M}$. This gives $f^*\lambda_{|0_{T^*M}} = 0$ and we get 
	\begin{equation*}
		dg_{|0_{T^*M}} = ( f^*\lambda - \lambda) _{|0_{T^*M}} = 0
	\end{equation*}
	and consequently, $g$ is constant on the zero-section $0_{T^*M}$, or in other words $ \osc (g_{|0_{T^*M}})=0$. Now, we know by assumption that $\lim\limits_n d_H(\mathcal{L}_n , \mathcal{L}) =0$, or by continuity of the Hamiltonian flow $\psi$, that $\lim\limits_n d_H(\psi^{-1}(\mathcal{L}_n) , 0_{T^*M}) =0$. Hence, uniform continuity of the map $g$ on a compact neighbourhood of the zero-section $0_{T^*M}$ yields the desired limit $\lim_n \osc (g_{|\psi^{-1}(\mathcal{L}_n)})$.
\end{proof}

\subsection{Structure of the Article}

The Section \ref{Extsection} of this article is devoted to the definition of an extended autonomous Hamiltonian $\mathscr{H}$ and an extended Lagrangian submanifold $\mathscr{L}$ in the new phase space $T^*(\mathbb{R} \times M)$. In Section \ref{GFsection}, we recall the main properties of generating functions and Liouville primitives and we construct convenient ones for $\mathscr{L}$. Section \ref{GSsection} contains a reminder on spectral invariants and graph selectors with the properties that will play a central role in the proof of the main theorem. A fitting graph selector $u$ of $\mathscr{L}$ will then be defined. The following Section \ref{Proofsection} is dedicated to the proof of the main results of this paper, using various concepts coming from the weak KAM Theory. An Appendix \ref{AppendixLO} introduces complementary tools that will serve to prove Corollary \ref{MainCor2}. \\

A knowledgeable reader about generating functions and spectral invariants can start with Section \ref{Extsection} and skip to Subsection \ref{GSConstruction} which deals with the construction of the extended graph selector.

\section{Extension of the Lagrangian Submanifold} \label{Extsection}

To have a global view over all the Lagrangian submanifolds $\mathcal{L}_t = \phi_H^t(\mathcal{L})$, we will consider time as part of the manifold variables and combine them into a single Lagrangian submanifold. Let $\mathcal{M} = \mathbb{R} \times M$ be a non-compact manifold with cotangent bundle $T^* \mathcal{M} = T^*\mathbb{R} \times T^*M$. We denote the coordinates in $T^* \mathcal{M}$ by $(\tau, E, q,p)$. The $1$-form $\Lambda = \lambda + Ed\tau$ is a Liouville form for $T^* \mathcal{M}$ endowed with the symplectic form $\Omega = -d\Lambda = \omega + d\tau \wedge dE$ \\

We extend the Hamiltonian $H : \mathbb{T}^1 \times T^*M \to \mathbb{R}$ to a Hamiltonian $\mathscr{H} : T^*\mathbb{R} \times T^*M \to \mathbb{R}$ defined by
\begin{equation}
	\mathscr{H}(\tau,E,q,p) = E + H( \tau, q, p)
\end{equation}

And we extend the exact Lagrangian submanifolds $\mathcal{L}_t$ to a Lagrangian submanifold of $T^*\mathcal{M}$ defined by
\begin{equation}
	\mathscr{L} := \big\{ \phi_\mathscr{H}^t\big(0,-H(0,q,p),q,p\big) \; | \; (q,p) \in \mathcal{L}, \; t \in \mathbb{R} \big\}
\end{equation}

Note that for $(q,p) \in \mathcal{L}$, $\mathscr{H}\big(0,-H(0,q,p), q,p\big) = -H(0,q,p) + H(0,q,p) =0$, and since the Hamiltonian is autonomous, we get the inclusion
\begin{equation} \label{zero}
	\mathscr{L} \subset \{ \mathscr{H} = 0 \}
\end{equation}

We will see in Proposition \ref{LExtExact} that the Lagrangian submanifold $\mathscr{L}$ is exact.

\section{Generating Functions} \label{GFsection}

This section is devoted to a brief presentation of generating functions and their main properties. Although everything presented here is known, we will give proofs of all properties except for Sikorav's existence theorem and Viterbo's uniqueness theorem.

Moreover, convenient generating functions for $\mathscr{L}$ will be constructed along the exposition.

\subsection{Generating Functions}

Let $\mathcal{M}$ be a manifold, not necessarily compact. Let $p : E \to \mathcal{M} $ be a finite-dimensional vector bundle over $\mathcal{M}$.

\begin{defi} \label{gf} 
	A $ \mathcal{C}^2$ map $S : E \to \mathbb{R}$ is a \textit{generating function} if
	\begin{enumerate}[label=\roman*.]
		\item Zero is a regular value of the map
	\begin{align*}
		E & \to  T^*E \\		
		(q,\xi)& \mapsto d_\xi S(q,\xi)
	\end{align*}
	So that the \textit{critical locus} $\Sigma_S = \{(q,\xi) \in E \; | \; d_\xi S(q,\xi) =0  \}$ is a submanifold of $E$.
	\item The map 
	\begin{equation}
		\begin{split}
			i_S : \; & \Sigma_S \to T^*\mathcal{M} \\
			& (q, \xi) \mapsto (q,d_q S(q,\xi) )
		\end{split}
	\end{equation}
	is a diffeomorphism from $\Sigma_S$ onto its image $\mathcal{L}_S = i_S( \Sigma_S )$. $\mathcal{L}_S$ is the (exact) Lagrangian submanifold generated by $S$.
	\end{enumerate}

	When the bundle $E$ is trivial i.e. $E = \mathcal{M} \times \mathbb{R}^k$ and there exists a non-degenerate quadratic form $Q : \mathbb{R}^k \to \mathbb{R}$ and a real constant $c \in \mathbb{R}$ such that $S = c + Q$ outside of a compact subset of $E$, we say that $S$ is a \textit{generating function quadratic at infinity} or \textit{g.f.q.i}. The \textit{index} of $S$ is the index of the quadratic form $Q$.
\end{defi}

\begin{theo} \label{Exist}
	(Sikorav's existence theorem \cite{MR0882965}) Let $\mathcal{L}$ be a Lagrangian submanifold of $T^* \mathcal{M}$ that admits a g.f.q.i and let $\phi^t$ be a Hamiltonian isotopy. Then there exists a smooth path of g.f.q.i $(S_t : M \times \mathbb{R}^k \to \mathbb{R})_{t \in [0,1]}$ with a fixed dimension $k$ such that for all $t \in [0,1]$, $S_t$ generates $\phi^t( \mathcal{L})$.
\end{theo}

This theorem proven in paragraph $1.7$ of \cite{MR0882965} has been simplified by M.Brunella in \cite{MR1115746}. Actually, the method exposed in \cite{MR1115746} does not provide a smooth path, but gives only the generating function $S_1$. An easy adaptation proves this parametrized version (see \cite{humil2008} Appendix B). A more general theorem by Théret \cite{MR1709692} states that the space of g.f.q.i is, up to equivalence, a smooth Serre fibration over the space of Lagrangian submanifolds Hamiltonianly isotopic to the zero section.\\

We can define for all time $T>0$ the paths $(S^T_t : M \times \mathbb{R}^{k_T} \to \mathbb{R})_{t \in [-T,T]}$ such that $S^T_t$ generates $\mathcal{L}_t := \phi_H^t( \mathcal{L})$. We may sometimes write $S^T(t,x)$ instead of $S^T_t(x)$.

\begin{defi}
	Let $S : E = \mathcal{M} \times \mathbb{R} \to \mathbb{R}$ be a generating function. The \textit{basic operations} on generating functions are
	\begin{itemize}
		\item \textit{(Translation)} $S' = S+c : E \to \mathbb{R}$ for some constant $c \in \mathbb{R}$.
		\item \textit{(Bundle isomorphism)} $S' = S \circ F : E' \to \mathbb{R}$ for some bundle isomorphism $F : E' \to E$ where $p' : E' \to \mathcal{M}$ is a vector bundle over $\mathcal{M}$.
		\item \textit{(Stabilization)} $S' = S \oplus Q' : E \oplus E' \to \mathbb{R}$ for some map $Q' : E' \to \mathbb{R}$ that is a non-degenerate quadratic form when restricted to the fibers of a finite dimensional vector bundle $p' : E' \to \mathcal{M}$.   
	\end{itemize}
	Two generating functions $S$ and $S'$ are said \textit{equivalent} if they can be made equal to a third generating function $S''$ after a succession of basic operations.
\end{defi}

\begin{theo} \label{Uniq}
	(Viterbo's uniqueness theorem \cite{MR1157321,MR1709692}) If $S$ and $S'$ are two g.f.q.i that generate the same Lagrangian submanifold $\mathcal{L}$, then they are equivalent.
\end{theo}

\subsection{Liouville Primitives} \label{LPrim}

Recall that an exact Lagrangian submanifold $\mathcal{L}$ in $T^*M$ is a Lagrangian submanifold such that the Liouville form restricted to it $\lambda_{|\mathcal{L}}$ is exact i.e. $\lambda_{|\mathcal{L}} = dh$ with Liouville primitive $h: \mathcal{L} \to \mathbb{R}$.\\

The Lagrangian submanifolds $\mathcal{L}_t$ are exact and do admit Liouville primitives that can be deduced from their generating functions as follows
\begin{equation}
	h^T_t = S^T_t \circ i_{S|\mathcal{L}_t}^{-1}
\end{equation}

\begin{prop}
	$h^T_t$ is a Liouville primitive on $\mathcal{L}_t$
\end{prop}

\begin{proof} For the sake of simplicity, we omit $T$ and $t$. Let $(q,p) \in \mathcal{L}$ and $(q, \xi) = i_{S|\mathcal{L}}^{-1}(q,p)$. By definition of $S$, we have that $p = d_qS$ and $d_\xi S(q, \xi) = 0$.
	\begin{align*}
		dh(q,p) & = d (S \circ i_{S|\mathcal{L}}^{-1}) (q,p) = dS(i_{S|\mathcal{L}}^{-1}(q,p)) \circ d i_{S|\mathcal{L}}^{-1} (q,p) \\
		&= \begin{pmatrix}
			d_qS(q, \xi) & d_\xi S(q,\xi)
		\end{pmatrix} 
		\begin{pmatrix}
			1 & 0 \\
			\times & \times
		\end{pmatrix}
		= \begin{pmatrix}
			p & 0
		\end{pmatrix} 
		\begin{pmatrix}
			1 & 0 \\
			\times & \times
		\end{pmatrix} \\
		&= p dq = \lambda(p,q)
	\end{align*}
\end{proof}

This gives us a path of Liouville primitives on $\mathcal{L}_t$. Now, we would like to construct a Liouville primitive on the Lagrangian submanifold $\mathscr{L}$.

Fix $h_0$ a Liouville primitive on $\mathcal{L}_0$. For all time $t \in \mathbb{R}$, define the map $h_t : \mathcal{L}_t \to \mathbb{R}$ by
\begin{equation} \label{Primitive}
	h_t(\phi_H^{t}(x)) = h_0(x) + \int_0^t \big( \gamma_x^* \lambda - H(\tau, \gamma_x(\tau)) \big) \; d\tau
\end{equation}
where $x$ belongs to $\mathcal{L}$ and $\gamma_x(\tau) = \phi_H^{\tau}(x)$.

\begin{prop} \label{htPrimitive}
	For all time $t \in \mathbb{R}$, $h_t$ is a Liouville primitive on $\mathcal{L}_t$.
\end{prop}

\begin{proof}
	Fix a positive time $t>0$. Let $x_t = \phi_H^t(x)$ and $y_t = \phi_H^t(y)$ be two points of $\mathcal{L}_t$ linked by a curve $\sigma_t = \phi_H^t(\sigma) : [0,1] \to \mathcal{L}_t$. Set $\gamma_x(\tau) = \phi_H^\tau(x)$ and its extended curve $\zeta_x(\tau) =  \big(\tau, E_x(\tau),\gamma_x(\tau)\big) \in \mathscr{L}$. Since $\mathscr{L} \subset \{ \mathscr{H} = 0 \}$, for all $\tau$, $\mathscr{H}\big(\zeta_x(\tau)\big)=0$ gives
	\begin{equation}
		E_x(\tau) = - H\big(\tau, \gamma_x(\tau)\big)
	\end{equation}
	Similarly, define $\gamma_y(\tau) = \phi_H^\tau(y )$ and its extension $\zeta_y : [0,t] \to \mathscr{L}$. We denote by $\sigma : [0,1] \to \mathscr{L}$ the extention of $\sigma : [0,1] \to \mathcal{L}$. Now set $\overline{\zeta_x}(\tau)=\zeta_x(t-\tau)$ the time-backward curve corresponding to $\zeta_x$ defined on $\tau \in [0,t]$, and $\zeta = \overline{\zeta_x}.\sigma. \zeta_y$ the concatenation of the three curves. Then we have
	\begin{align*}
		h_t(y_t) - h_t(x_t) & = - \int_0^t \big( \gamma_x^* \lambda - H(\tau, \gamma_x(\tau))\big) \; d\tau + [h_0(y) - h_0(x)] + \int_0^t \big( \gamma_y^* \lambda - H(\tau, \gamma_y(\tau)) \big) \; d\tau\\
		&= \int_{\overline{\zeta_x}} \Lambda + \int_\sigma \Lambda + \int_{\zeta_y} \Lambda = \int_\zeta \Lambda = \int_{\phi_\mathscr{H}^{[0,t]}(\sigma)} d\Lambda + \int_{\sigma_t} \Lambda = 0 + \int_{\sigma_t} \Lambda = \int_{\sigma_t} \lambda 
	\end{align*}
	where we used Stokes formula on the set $\phi_\mathscr{H}^{[0,t]}(\sigma) := \bigcup_{\tau \in [0,1]} \phi_\mathscr{H}^\tau(\sigma)$, and we used the fact that $\mathscr{L}$ is a Lagrangian submanifold of $(T^* \mathcal{M}, -d\Lambda)$.
\end{proof}

The wanted extended Liouville primitive is given by the map $\mathpzc{h} : \mathscr{L} \to \mathbb{R}$ defined by
\begin{equation}
	\mathpzc{h}(t,E,q,p) = h_t(q,p)
\end{equation}
where $E$ is such that $(t,E,q,p) \in \mathscr{L}$.\\

\begin{prop} \label{LExtExact}
	$\mathpzc{h}$ is a Liouville primitive on $\mathscr{L}$.
\end{prop}

We omit the proof of this assertion as it is contained in the proof of the Proposition \ref{htPrimitive} above.

\begin{rem}
	Note that for $s <t$ be two real times and $x \in \mathcal{L}_s$, if we set $\gamma(\tau) = \phi_H^{s,\tau}(x)$, then
	\begin{equation} \label{Prim}
		h_t(\phi_H^{s,t}(x)) = h_s(x) + \int_s^t \big( \gamma^* \lambda - H(\tau, \gamma(\tau)) \big) \; d\tau 
	\end{equation}
\end{rem}

\hfill

In the following, we will link the primitive $\mathpzc{h}$ to the constructed primitives $h^T_t$ and their generating functions $S^T_t$. Fix $T > |t|$, we know that $h_t$ and $h^T_t$ are two Liouville primitives for $\mathcal{L}_t$, hence they differ by a constant $\delta^T_t$ and $h_t = h^T_t +  \delta^T_t$. This constant depends smoothly on $t$ since the paths $(h^T_t)_t$ and $(h_t)_t$ are smooth. Moreover, for the two generating functions $S^T_t$ and $S^T_t+\delta_t$, we have $i_{S^T_t} = i_{S^T_t + \delta_t}$.\\ 

Hence, by replacing $S^T_t$ by the g.f.q.i $S^T_t + \delta_t$, we can assume that for all $T>0$ and $t \in [-T,T]$, 
\begin{equation}\label{ntransl}
	h_t = h^T_t = S^T_t \circ i_{S|\mathcal{L}_t}^{-1}
\end{equation}

From now on, this will be assumed true.

\begin{prop} \label{GFext}
	Let $T>0$ be a fixed real time. Then, with the notation $\mathscr{L}_{(-T,T)} = \mathscr{L} \cap \{ \tau \in (-T,T) \}$ we have the equality
	\begin{equation}
		\mathscr{L}_{(-T,T)} = \big\{\big(\tau, \partial_\tau S^T(\tau,q,\xi), q, d_q S^T(\tau, q,\xi) \big) \; | \; d_\xi S^T(\tau,q,\xi) = 0, \; (\tau,q) \in (-T,T) \times M \big\}
	\end{equation}
\end{prop}

In other words, the map $S^T$ plays the role of a g.f.q.i for $\mathscr{L}_{(-T,T)}$ in $T^* \mathcal{M} = T^* (\mathbb{R} \times M)$.

\begin{proof}
	Let $\tau \in (-T,T)$ and let $(q, p)$ be a point of $\mathcal{L}_\tau$. By definition of the generating functions, we have that $p = d_q S^T( \tau, q, \xi)$ with $(q,\xi) = i_{S^T_\tau}^{-1} (q,p)$. Let $E \in \mathbb{R}$ be the energy such that $(\tau,E,q,p) \in \mathscr{L}$. We know from Theorem \ref{Exist} that $S^T$ is regular with respect to $\tau$ and thus a derivation is possible. We need to prove that $E = \partial_\tau S^T( \tau, q, \xi)$.
	
	We define the set $\Sigma_{S^T}$ by 
	\begin{equation}
		\Sigma_{S^T} := \big\{(\tau, q,\xi) \in (-T,T) \times M \times \mathbb{R}^{k_T} \; | \; d_\xi S^T(\tau, q, \xi) =0 \big\}
	\end{equation}
	so that the map $S^T(\tau,q,\xi) = h_\tau \circ i_{S^T}(\tau,q,\xi)$ is defined on $\Sigma_{S^T}$. We need to differentiate with respect to $\tau$. To do so, we determine $T_{(\tau,q,\xi)}\Sigma_{S^T}$ given by differentiating the equation $d_\xi S^T(\tau, q, \xi) =0$. Hence,
	\begin{equation}
		T_{(\tau,q,\xi)}\Sigma_{S^T} = \{ (\delta \tau, \delta q, \delta \xi) \; | \; d^2_{\tau\xi} S^T(\tau,q,\xi). \delta \tau + d^2_{q\xi} S^T(\tau,q,\xi). \delta q + d^2_{\xi\xi} S^T(\tau,q,\xi). \delta \xi =0 \}
	\end{equation}
	Moreover, since $S^T_\tau$ is a generating function of $\Sigma_{S^T_\tau}$, we know that the map $(\delta q , \delta \xi ) \mapsto d^2_{q\xi} S^T(\tau,q,\xi). \delta q + d^2_{\xi\xi} S^T(\tau,q,\xi). \delta \xi$ is onto. Hence, for all $\delta\tau$, there exist $(\delta q,\delta \xi)$ such that $(\delta \tau,  \delta q, \delta \xi) \in T_{(\tau,q,\xi)}\Sigma_{S^T}$.
	
	Take such a vector $(\delta \tau,  \delta q, \delta \xi) \in T_{(\tau,q,\xi)}\Sigma_{S^T}$. We have
	\begin{align*}
		dS^T(t,q,\xi).(\delta \tau,  \delta q, \delta \xi) &= \partial_\tau S^T(\tau,q,\xi). \delta \tau + d_q S^T(\tau,q,\xi). \delta q + d_\xi S^T(\tau,q,\xi). \delta \xi \\
		&= \partial_\tau S^T(\tau,q,\xi). \delta \tau + p\delta q
	\end{align*}
	and
	\begin{align*}
		dS^T(t,q,\xi).(\delta \tau,  \delta q, \delta \xi) &= d(h^T \circ i_{S^T})(t,q,\xi).(\delta \tau,  \delta q, \delta \xi)\\
		&= dh^T(t,q,p) \circ di_{S^T}(t,q,\xi) .(\delta \tau,  \delta q, \delta \xi) \\
		&= \Lambda(t,q,p) \circ di_{S^T}(t,q,\xi) .(\delta \tau,  \delta q, \delta \xi)
		\\&= \begin{pmatrix}
			E & p & 0
		\end{pmatrix} 
		\begin{pmatrix}
			1 & 0 & 0 \\
			0 & 1 & 0 \\
			d^2_{tq}S^T(t,q,\xi) & d^2_{qq}S^T(t,q,\xi) & 0
		\end{pmatrix}
		\begin{pmatrix}
			\delta \tau \\
			\delta q \\
			\delta \xi
		\end{pmatrix} \\
		&= \begin{pmatrix}
			E & p & 0
		\end{pmatrix} 
		\begin{pmatrix}
			\delta \tau \\
			\delta q \\
			\delta \xi
		\end{pmatrix}
		= E \delta \tau + p \delta q
	\end{align*}
	where we use that $h$ derives from $\mathpzc{h}$, Liouville primitive on $\mathscr{L}$. We obtain
	\begin{equation*}
		(\partial_\tau S^T(\tau,q,\xi)-E) \delta \tau = p\delta q - p\delta q = 0
	\end{equation*}    
	Taking $\delta \tau \neq 0$, which is possible by the study of $T_{(\tau,q,\xi)}\Sigma_{S^T}$, we conclude that $\partial_t S^T(t,q,\xi) = E$. Therefore, $(t,E,q,p) \in \mathscr{L}$ if and only if $p = d_q S^T( t, q, \xi)$ and $E = \partial_t S^T(t,q,\xi)$ with $(q,\xi) = i_{S^T_t}^{-1} (q,p)$.
\end{proof}

\section{Spectral Invariants and Graph Selectors} \label{GSsection}

In this section we present the construction and properties of spectral invariants rising from generating functions. Subsections \ref{SIsubsection} and \ref{SI2} contain the construction and properties of these spectral invariants. And in Subsection \ref{GSsubsection} we introduce the notion of graph selectors which will play a fundamental role in the proof of the main result.\\

All material in this section is known and considered standard. Nevertheless, we chose to provide the proofs of all the properties needed in the final proof of this paper, following following \cite{MR1157321}. This may provide sufficient background for the reader unfamiliar with spectral invariants and graph selectors.

\subsection{Spectral Invariants} \label{SIsubsection}

In this section, we define the spectral invariants of functions $S:E = \mathcal{M} \times \mathbb{R}^k \to \mathbb{R}$ that are quadratic at infinity, i.e such that there exist a real constant $c \in \mathbb{R}$ and a non-degenerate quadratic form $Q : \mathbb{R}^k \to \mathbb{R}$ such that $S = c + Q$ outside of a compact set of $E$. We will denote f.q.i for functions quadratic at infinity.\\

Fix such an f.q.i $S$ with associated pair $(c,Q)$. The domain $\mathbb{R}^k$ of the non-degenerate quadratic form $Q$ decomposes into $F^+ \oplus F^-$ which are the positive and negative spaces for $Q$ and $m= \dim F^-$ is the index of $Q$.\\

For any real number $a \in \mathbb{R}$, we denote the sub-level of $S$ with height $a$ by
\begin{equation}
	S^a = \{ e \in E \; | \; S(e) \leq a \}
\end{equation}

Since $S= c+Q$ outside of a compact set, there exists a large $N>0$ such that $S^N = \mathcal{M} \times Q^{N-c}$ and $S^{-N} = \mathcal{M} \times Q^{-N-c}$. Hence, for any $a \in \mathbb{R}$ and large $b \geq N$, the homotopy types of the pairs $(S^b,S^a)$ and $(S^a, S^{-b})$ are independent of $b$. We denote these respectively by $(S^\infty,S^a)$ and $(S^a, S^{-\infty})$.\\

Let $H^*$ denote the simplicial cohomology with coefficients in a field, namely $\mathbb{R}$. By a successive application of K\"unneth formula and Thom's isomorphism to the trivial fibre bundle $\mathcal{M} \times F^- \to \mathcal{M}$, one gets
\begin{equation} \label{Thom}
		H^*(S^\infty, S^{- \infty}) \simeq H^*( \mathcal{M} ) \otimes H^*(Q^\infty, Q^{-\infty}) \simeq H^*( \mathcal{M} ) \otimes H^*(\mathbb{D}^m, \partial \mathbb{D}^m)  \simeq H^{*-m}( \mathcal{M} )
\end{equation}
where $\mathbb{D}^m$ is the $m$-dimensional disk.\\

Moreover, the inclusion $S^a \hookrightarrow S^b$ for $a<b$ induces the cohomological morphism $ H^*(S^b, S^{- \infty}) \hookrightarrow H^*(S^a, S^{- \infty})$. The existence of this map added to the fact that $H^*(S^{- \infty}, S^{- \infty}) = 0$ makes sense of the following definition.

\begin{defi} 
	Let $S:E = \mathcal{M} \times \mathbb{R}^k \to \mathbb{R}$ be a f.q.i. For all $\alpha \in H^*(\mathcal{M}) \setminus \{0\}$ we define the spectral invariant $c(\alpha,S)$ by
	\begin{equation} \label{SpectralInvariantsDef}
		c(\alpha,S) := \inf \{ a \in \mathbb{R} \; | \; \alpha \neq 0 \text{ in } H^*(S^a, S^{- \infty} ) \}
	\end{equation}
\end{defi}

\begin{rem}
	For all real numbers $a<b<c$, the inclusions $S^a \hookrightarrow S^b \hookrightarrow S^c$ yield the commutative diagram  
	\begin{equation}
		\begin{tikzcd}
			H^*(S^c, S^{- \infty} )  \arrow{r}  \arrow{rd} & H^*(S^b, S^{- \infty} )  \arrow{d} \\
   			& H^*(S^a, S^{- \infty} ) 
		\end{tikzcd}
	\end{equation}
	so that if $\alpha_q$ is null in $H^*(S^b, S^{- \infty})$, then it is null in $H^*(S^a, S^{- \infty})$ for all $a<b$, and we have
	\begin{equation}
		c(\alpha,S) = \sup \{ a \in \mathbb{R} \; | \; \alpha = 0 \text{ in } H^*(S^a, S^{- \infty} ) \}
	\end{equation}           
\end{rem}         

\begin{prop} \label{Crit}
	For any f.q.i $S$, $c(\alpha, S)$ is a critical value of $S$.
\end{prop}

\begin{proof}
	Assume $c = c(\alpha,S)$ is not a critical value of $S$. Since $S$ is equal to a non-degenerate quadratic form outside of a compact set, its critical points must be contained in a compact set. Hence there exists a small $\varepsilon > 0$ such that $S$ has no critical values in $[c- \varepsilon, c+\varepsilon]$ and the gradient of $S$ is non-null in $S^{c+\varepsilon} \setminus \mathring{S}^{c - \varepsilon } = S^{-1}([c- \varepsilon, c+\varepsilon])$. Therefore, $S^{c+\varepsilon}$ can be contracted to $S^{c - \varepsilon }$ using a gradient flow generated by a vector field of the form $X = - \chi \frac{\nabla S}{\Vert \nabla S\Vert ^2}$ where $\chi$ is a bump map identically equal to zero in $S^{c+\varepsilon} \setminus S^{c - \varepsilon }$. This shows that $H^*(S^{c+\varepsilon}, S^{c - \varepsilon }) = 0$.\\
	The long exact sequence for the triple $(S^{c+\varepsilon}, S^{c - \varepsilon }, S^{- \infty}) $ is
	\begin{equation}
		H^*(S^{c+\varepsilon}, S^{c - \varepsilon }) \rightarrow H^*(S^{c-\varepsilon}, S^{- \infty}) \rightarrow H^*(S^{c+\varepsilon}, S^{- \infty}) \rightarrow H^{*+1}(S^{c+\varepsilon}, S^{c - \varepsilon })
	\end{equation}
	with the two extremities being null. This results to the isomorphism
	\begin{equation} \label{isom}
		H^*(S^{c-\varepsilon}, S^{- \infty}) \simeq H^*(S^{c+\varepsilon}, S^{- \infty})
	\end{equation}
	However, by the definition of $c$ we have $\alpha \neq 0$ in $H^*(S^{c+\varepsilon}, S^{- \infty})$ and $\alpha = 0$ in $H^*(S^{c-\varepsilon}, S^{- \infty})$. We get a contradiction to (\ref{isom}).
\end{proof}

As an application to generating functions, we obtain the following proposition.
\begin{prop} \label{inegh}
	Let $\mathcal{L}$ be an exact Lagrangian submanifold of $T^* \mathcal{M}$ with g.f.q.i $S$ and Liouville primitive $h=S\circ i_{S|\mathcal{L}}^{-1}$. Then for all $\alpha \in H^*( \mathcal{M})$,
	\begin{equation}
		\min h \leq c(\alpha, S) \leq \max h
	\end{equation}
\end{prop}

\begin{proof}
	We know from Proposition \ref{Crit} that $c(\alpha, S)$ is a critical value of $S$. Then there exists $(q, \xi)$ such that $S(q,\xi) = c(\alpha, S)$ and $dS(q,\xi) =0$. It follows that $d_\xi S(q,\xi) =0$, $(q,\xi) \in \Sigma_S$ and
	\begin{align*}
		h(x) = S \circ i_S^{-1}(x) = S(q,\xi) = c(\alpha, S)
	\end{align*}
	Therefore we conclude that $\min h \leq c(\alpha, S) \leq \max h$.
\end{proof}

For any f.q.i $S$, we denote by $\Vert S\Vert _\infty$ the quantity in $[0,+\infty]$ given by
\begin{equation}
	\Vert  S \Vert _\infty = \sup \{ |S(q,\xi)| \; | \; (q, \xi) \in M \times \mathbb{R}^k \}
\end{equation}

\begin{prop} \label{Cont}
	For any f.q.i $S_1$, there exists a real number $\varepsilon_0 >0$ small enough such that for all $0<\varepsilon< \varepsilon_0$ and all f.q.i $S_2$ such that $\Vert S_2 - S_1 \Vert_\infty \leq \epsilon $, we have $|c(\alpha,S_2) - c(\alpha,S_1)| \leq \varepsilon$.
\end{prop}

\begin{proof}
	To simplify the notations, set $c_1 =c(\alpha,S_1)$. For $\varepsilon$ small enough, the non-degenerate quadratic forms at infinity associated to $S_1$ and $S_2$ have same index. Hence, we infer that $S_1^{\pm\infty}$ and $S_2^{\pm\infty}$ are homotopically equivalent and we obtain an isomorphism $H^*(S_1^{\infty},S_1^{-\infty}) \simeq H^*(S_2^{\infty},S_2^{-\infty})$ sending the element $\alpha_1$ corresponding to $\alpha$ in $H^*(S_1^{\infty},S_1^{-\infty})$ to the element $\alpha_2$ corresponding to $\alpha$ in $H^*(S_2^{\infty},S_2^{-\infty})$.
	
	 Now fix a real number $\delta >0$. We know from the hypothesis that $S_2 \leq S_1 + \varepsilon$ and more precisely that $S_1^{c_1+\delta} \subset S_2^{c_1+\delta+\varepsilon}$. This inclusion yields the cohomological morphism 
	\begin{equation}
		i^*: H^*(S_2^{c_1+\delta+\varepsilon}, S_2^{-\infty}) \rightarrow H^*(S_1^{c_1+\delta}, S_1^{- \infty})
	\end{equation}
	and we obtain a commutative diagram
	\begin{equation}
		\begin{tikzcd}
			H^*(S_2^{\infty},S_2^{-\infty}) \arrow{d} \arrow[r,"\sim"]  & H^*(S_1^{\infty},S_1^{-\infty}) \arrow{d} \\
			H^*(S_2^{c_1+\delta+\varepsilon}, S_2^{-\infty}) \arrow{r}{i^*}  & H^*(S_1^{c_1+\delta}, S_1^{- \infty})
		\end{tikzcd}	
	\end{equation}	
	By definition of $c_1$, we know that $i^*(\alpha_2) = \alpha_1 \neq 0$ in $H^*(S_1^{c_1+\delta}, S_1^{- \infty})$. Thus, $\alpha_2 \neq 0$ in $H^*(S_2^{c_1+\delta+\varepsilon}, S_2^{-\infty})$ meaning that $c(\alpha,S_2) \leq c_1+\delta+\varepsilon$. By letting $\delta$ go to zero, we conclude that
	\begin{equation}
		c(\alpha,S_2) - c(\alpha,S_1) \leq \varepsilon
	\end{equation}
	The symmetry between $S_1$ and $S_2$ yields the inverse inequality.
\end{proof}

\subsection{Operations on Spectral Invariants} \label{SI2}

We focus in this subsection on two key operations on Lagrangian submanifolds and their induced properties on spectral invariants. These will show crucial in the proof of the main theorem.

\begin{defi}
	\begin{enumerate}
		\item For a Lagrangian submanifold $\mathcal{L}$ in $T^* \mathcal{M}$ we define its inverted Lagrangian submanifold $\overline{\mathcal{L}}$ as
		\begin{equation} \label{bar}
			\overline{\mathcal{L}} := \{ (q,-p) \; | \; (q,p) \in \mathcal{L} \}
		\end{equation}
		If $S$ is a generating function of $\mathcal{L}$, $-S$ is a generating function of $\overline{\mathcal{L}}$.
		\item For two Lagrangian submanifolds $\mathcal{L}_1$ and $\mathcal{L}_2$ in $T^* \mathcal{M}$ we define their fibred sum $\mathcal{L}_1 \string#  \mathcal{L}_2$ as the set
		\begin{equation}
			\mathcal{L}_1 \string#  \mathcal{L}_2 := \{ (q , p_1+p_2) \; | \; (q,p_1) \in \mathcal{L}_1, \; (q,p_2) \in \mathcal{L}_2 \}
		\end{equation}
		If $S_1(q,\xi_1)$ and $S_2(q,\xi_2)$ are two respective generating functions of $\mathcal{L}_1$ and $\mathcal{L}_2$, then we associate to $\mathcal{L}_1 \string#  \mathcal{L}_2$ the fibred sum $S_1 \oplus S_2$ given by
		\begin{equation}
			S_1 \oplus S_2(q,\xi_1,\xi_2) := S_1(q,\xi_1) + S_2(q,\xi_2)
		\end{equation}
	\end{enumerate}
\end{defi}

\begin{rem}
	Note that $\mathcal{L}_1 \string#  \mathcal{L}_2$ is not necessarily a submanifold and hence not a Lagrangian submanifold of $T^* \mathcal{M}$, but the fibred sum $S : = S_1 \oplus S_2$ still generates $\mathcal{L}_1 \string#  \mathcal{L}_2$ in the sense that
	\begin{equation*}
		\mathcal{L}_1 \string#  \mathcal{L}_2 = \{ (q, d_qS(q,\xi_1,\xi_2) \; | \; d_{(\xi_1,\xi_2)}S(q,\xi_1,\xi_2) =0 \}
	\end{equation*}
\end{rem}

The function $S_1 \oplus S_2$ is not quadratic at infinity. This issue is solved by the following proposition.

\begin{prop} \label{FibredGFQI}  
	Let $\mathcal{L}_1$ and $\mathcal{L}_2$ be two Lagrangian submanifolds generated by g.f.q.i $S_1$ and $S_2$ with corresponding constant and quadratic form pairs $(c_1,Q_1)$ and $(c_2,Q_2)$. Then, $\mathcal{L}_1 \string#  \mathcal{L}_2$ is generated by a g.f.q.i $S$ with corresponding pair $(c_1+c_2,Q_1 \oplus Q_2)$, and there exists a diffeomorphism $\phi$ of $T^*\mathcal{M}$ such that $S \circ \phi = S_1 \oplus S_2$.
\end{prop}       
\begin{proof}
	For simplicity, we assume that $c_1 = c_2 = 0$. We set $S^0:= S_1 \oplus S_2$, $Q:= Q_1 \oplus Q_2$. We have
	\begin{equation*}
		S^0-Q = S_1 \oplus S_2 - Q_1 \oplus Q_2 = (S_1 - Q_1) \oplus (S_2 - Q_2)
	\end{equation*}
	with the supports of $S_1 - Q_1$ and $S_2 - Q_2$ being both compact in their domains of definition. Hence, there exists a constant $C_1$ and $C_2>0$ such that 
	\begin{equation*}
		\Vert S_1-Q_1 \Vert_1 := \Vert S_1-Q_1 \Vert_\infty + \Vert \nabla( S_1-Q_1 )\Vert_\infty \leq C_1 \quad \text{and} \quad \Vert S_2-Q_2 \Vert_1 \leq C_2
	\end{equation*}
	and setting $C = C_1 + C_2$, we get
	\begin{equation} \label{FibredGFQIDem1}
		\Vert S^0-Q \Vert_1 \leq \Vert S_1-Q_1 \Vert_1  + \Vert S_2-Q_2 \Vert_1  \leq C1+C_2 = C
	\end{equation}
	
	Let $B>A>0$ be large real numbers and consider an increasing smooth map $\rho : [0, +\infty) \to [0,1]$ such that $\rho \equiv 0$ on $[0,A]$, $\rho \equiv 1$ on $[B+\infty)$ and $|\rho'| \leq \varepsilon$ small enough. We consider the map
	\begin{equation*}
		S^1(q,\xi) = \rho(\xi)Q(q,\xi) + (1-\rho(\xi))S^0(q,\xi)
	\end{equation*}
	and the family $(S^t)_{t \in [0,1]}$ defined by
	\begin{equation*}
		S^t = (1-t)S^0 + tS^1 = Q + (1-t\rho)(S^0 - Q)
	\end{equation*}
	Note that $S^1$ is quadratic at infinity with quadratic form $Q$ on $\{|\xi| \geq B\}$.\\
	
	We will construct an isotopy $\phi^t$ such that for all $t \in [0,1]$, $S^t \circ \phi^t = S^0$. We consider the vector flow $X_t = \frac{d\phi^t}{dt} \circ (\phi^t)^{-1}$ associated to $\phi^t$. We assume that $\phi^t$ preserves the fibres and we adopt the notations $\phi^t(q,\xi) = (q,\eta)$ and $X_t = (0, \sigma_t)$. 
	We would like to get
	\begin{equation}
		\frac{d}{dt} (S^t \circ \phi^t) = 0
	\end{equation}
	Hence, computation yields
	\begin{align*}
		0 &= \frac{dS^t}{dt}(q,\eta) + dS^t.X_t(q,\eta) \\
		&= \rho(\eta) (S^0 -Q)(q,\eta) + \langle  \partial_\eta S^t(q,\eta) , \sigma_t(q,\eta) \rangle
	\end{align*}
	with 
	\begin{align*}
		\partial_\eta S^t(q,\eta) = \partial_\eta Q(q,\eta) - t\rho'(\eta)(S^0-Q)(q,\eta) + (1-t\rho)\partial_\eta(S^0 - Q)
	\end{align*}
	
	We set $\sigma_t = 0$ for $|\eta| \leq A$. For $|\eta| \geq A$, we verify that we can define $\sigma_t$ such that
	\begin{equation} \label{FibredGFQIDem2}
		\langle  \partial_\eta S^t(q,\eta) , \sigma_t(q,\eta) \rangle = -  \rho(\eta) (S^0 -Q)(q,\eta)
	\end{equation}	
		
	Since the quadratic form $Q$ is non-degenerate, we deduce that there exists a constant $C'>0$ such that $|\partial_\eta Q| \geq C'|\eta|$. And using \eqref{FibredGFQIDem1}, we get for $|\eta| \geq A$
	\begin{align*}
		| \partial_\eta S^t(q,\eta) | &\geq |\partial_\eta Q(q,\eta)| - \big| t\rho'(\eta)(S^0-Q)(q,\eta) \big| - \big| (1-t\rho)\partial_\eta(S^0 - Q) \big| \\
		&\geq C'|\eta| - C(\varepsilon+1) \geq C'A - C(\varepsilon+1) \geq 2C'A
	\end{align*}
	for $A$ large enough. Therefore, the identity \eqref{FibredGFQIDem2} allows to define $\sigma_t$ and we have
	\begin{align*}
		|\sigma_t(q,\eta)| \leq \frac{|-  \rho(\eta) (S^0 -Q)(q,\eta) | }{|\partial_\eta S^t(q,\eta)|} \leq \frac{C}{2C'A}
	\end{align*}
	which shows that the vector field $X_t$ is complete. The diffeomorphism $\phi := \phi^1$ and the g.f.q.i $S := S^1$ verify the desired properties of the statement. 
\end{proof}

\begin{rem} 
	The proposition above implies that the level sets of $S_1 \oplus S_2$ and of the g.f.q.i $S$ are homotopic. Consequently we get the homotopy equivalence 
	\begin{equation}
		\big((S_1 \oplus S_2)^\infty, (S_1 \oplus S_2)^{-\infty} \big) \simeq \big((Q_1 \oplus Q_2)^\infty, (Q_1 \oplus Q_2)^{-\infty} \big)
	\end{equation}
	where $Q_1 \oplus Q_2$ is quadratic on fibres with index the sum of the indexes of $S_1$ and $S_2$.
	
	This means that $S_1 \oplus S_2$ can be seen as a f.q.i and with well defined spectral invariants $c(\alpha,S_1 \oplus S_2)$ as in \eqref{SpectralInvariantsDef}. We can that it plays the role of a g.f.q.i for the set $\mathcal{L}_1 \string#  \mathcal{L}_2$.
\end{rem}         

We now focus on the consequences of these two operations on the spectral invariants. We start by a property on the inversion operation.

\begin{prop} \label{propc} 
	Let $1$ and $\mu$ be the respective generators of $H^0(\mathcal{M})$ and $H^d( \mathcal{M})$ and let $S: E = \mathcal{M} \times \mathbb{R}^k \to \mathbb{R}$ be a q.f.i. Then, we have
	\begin{equation}
		c(\mu, -S) = - c(1, S)
	\end{equation}
\end{prop}

\begin{proof}
	we first link the sublevels of the g.f.q.i $S$ to those of the g.f.q.i $-S$. Let $a$ be a real number.
	\begin{equation}
		(-S)^a = \{ x \in E \; | \; -S(x) \leq a \} = E \setminus S^{-a}
	\end{equation}
	Thus, by Alexander duality, we have the isomorphism
	\begin{equation}
		AD : H_{m}(S^{-a}, S^{- \infty}) \overset{\sim}{\longrightarrow} H^{d+k-m}((-S)^a , (-S)^{-\infty})
	\end{equation}
	where $m$ is the index of $S$. And writing (vertically) the exact homological and cohomological sequences of the triplets $(S^{\infty}, S^{-a}, S^{-\infty})$ and $((-S)^\infty,(-S)^a,(-S)^{-\infty})$, we get the following commutative diagram
	\begin{equation} \label{diag1}
		\begin{tikzcd}
			H_{m}(S^{-a}, S^{-\infty}) \arrow{d} \arrow[r,"\sim"]  & H^{d+k-m}((-S)^\infty, (-S)^a)\arrow{d} \\
			H_{m}(S^{\infty}, S^{-\infty}) \arrow{d} \arrow[r, "\sim"]  & H^{d+k-m}((-S)^\infty, (-S)^{- \infty}) \arrow{d} \\
			H_{m}(S^{\infty}, S^{-a}) \arrow[r, "\sim"] & H^{d+k-m}((-S)^a, (-S)^{-\infty})
		\end{tikzcd}	
	\end{equation}
	But before proceeding to a diagram chasing, we mention that the isomorphism (\ref{Thom}) and its homological counterpart lead to
	\begin{equation}
		H^{d+k-m}((-S)^\infty, (-S)^{- \infty}) \simeq H^d( \mathcal{M} ) \quad \text{and} \quad H_{m}(S^\infty, S^{- \infty}) \simeq H_0( \mathcal{M} )
	\end{equation}
	Thus, we can see $\mu \in H^d(\mathcal{M})$ and its Poincaré dual $1_0 \in H_0( \mathcal{M} )$ as respective elements of $H^{d+k-m}((-S)^\infty, (-S)^{- \infty})$ and $H_{m}(S^\infty, S^{- \infty})$. And with naturality $AD(1_0) = \mu$.
	
	Suppose that $a < -c(1, S)$. Since $-a > c(1, S)$ and by definition of the spectral invariants, we know that $1 \neq 0$ in $H^{m}(S^{-a}, S^{- \infty})$. Then the morphism $ H^{m}(S^{\infty}, S^{- \infty}) \rightarrow H^{m}(S^{-a}, S^{- \infty}) $ is non null and so is its transpose map
	\begin{equation}
		H_{m}(S^{-a}, S^{- \infty}) \longrightarrow H_{m}(S^{\infty}, S^{- \infty})
	\end{equation}
	And since $\dim H_{m}(S^{\infty}, S^{- \infty}) =1$, the observation above means that $1_0 \in H_{m}(S^{-a}, S^{- \infty})$ and $AD(1_0) \neq 0$ in $H^{d+k-m}((-S)^\infty, (-S)^a)$. By exactness of the vertical lines in the diagram, we get that $\mu =  AD(1_0) = 0$ in $H^{d+k-m}((-S)^a, (-S)^{-\infty})$ and hence $a \leq c(\mu, -S)$. As a result, 
	\begin{equation}\label{ineg1}
		-c(1, S) \leq  c(\mu, -S)
	\end{equation}
	
	Now, suppose that $a < c(\mu, -S)$. We know that $\mu = 0$ in $H^{d+k-m}((-S)^a, (-S)^{-\infty})$, then $\mu$ belongs to
	\begin{multline} \label{exac}
		\ker \big(H^{d+k-m}((-S)^\infty, (-S)^{- \infty}) \to H^{d+k-m}((-S)^a, (-S)^{-\infty}) \big) \\ 
		= \text{Im } \big(  H^{d+k-m}((-S)^\infty, (-S)^a) \to H^{d+k-m}((-S)^\infty, (-S)^{- \infty}) \big)
	\end{multline}
	And since $\mu \neq 0$ in $H^{d+k-m}((-S)^\infty, (-S)^{- \infty})$, the second morphism of (\ref{exac}) is non null and so is its left counterpart in the diagram (\ref{diag1}). Hence, the transpose map
	\begin{equation}
		H^{m}(S^{\infty}, S^{- \infty}) \rightarrow H^{m}(S^{-a}, S^{- \infty})
	\end{equation}
	is non null. And since $\dim H^{m}(S^{\infty}, S^{- \infty}) =1$, its generator $1$ has a non null image in $H^{m}(S^{-a}, S^{- \infty})$. Therefore, $-a \geq c(1, S)$ and $a \leq -c(1, S)$. We obtain 
	\begin{equation} \label{ineg2}
		c(\mu, -S) \leq -c(1, S)
	\end{equation}
	The inequalities (\ref{ineg1}) and (\ref{ineg2}) give the desired property.
\end{proof}

\begin{prop} \label{Invariance}
	Let $\mathcal{L}_1$ and $\mathcal{L}_2$ be two exact Lagrangian submanifolds of $T^* \mathcal{M}$. Let $\phi^t$ be a Hamiltonian flow and let $S_{1,t}$ and $S_{2,t}$ be two (smooth) paths of generating functions respectively for $\phi^t(\mathcal{L}_1)$ and $\phi^t(\mathcal{L}_2)$ such that their associated Liouville primitives $h_{1,t}$ and $h_{2,t}$ do verify the identity \eqref{Primitive}. Then for all $\alpha \in H^*( \mathcal{M} ) $, the map $t \mapsto c \left( \alpha, S_{1,t} \ominus S_{2,t} \right)$ is constant for $S_{1,t} \ominus S_{2,t} := S_{1,t} \oplus (-S_{2,t})$.
\end{prop}

\begin{proof}
	From Proposition \ref{Crit}, we know that $c \left( \alpha, S_{1,t} \ominus S_{2,t} \right)$ is a critical value of $S_{1,t} \ominus S_{2,t}$. Then there exists a critical point $(q_t,\xi_t, \eta_t) \in \mathcal{M} \times \mathbb{R}^{k_{S_1}} \times \mathbb{R}^{k_{S_2}}$ of $S_{1,t} \ominus S_{2,t}$ such that
	\begin{equation}
		c \left( \alpha, S_{1,t} \ominus S_{2,t} \right) = S_{1,t}(q_t, \xi_t) - S_{2,t}(q_t,\eta_t)
	\end{equation}
	and 
	\begin{equation}
		d_\xi S_{1,t}(q_t, \xi_t) =0, \quad d_\eta S_{2,t}(q_t,\eta_t), \quad d_q S_{1,t}(q_t, \xi_t) = d_q S_{2,t}(q_t,\eta_t)
	\end{equation}
	Set $x_t=\big(q_t,d_q S_{1,t}(q_t, \xi_t)\big) = \big(q_t,d_q S_{2,t}(q_t,\eta_t)\big) \in \phi^t(\mathcal{L}_1) \cap \phi^t(\mathcal{L}_2) =\phi^t(\mathcal{L}_1 \cap \mathcal{L}_2) $ and $x=(q,p) \in T^*\mathcal{M}$ such that $\gamma(t) := x_t = \phi^t(x)$.  Then using (\ref{Primitive}), we get
	\begin{align*}
		c \left( \alpha, S_{1,t} \ominus S_{2,t} \right) & = h_{1,t}(x_t) - h_{2,t}(x_t) \\
		& = \big[ h_{1,0}(x) + \int_0^t \gamma^* \lambda - H(\tau, \gamma(\tau))\; dt \big] - \big[ h_{2,0}(x) + \int_0^t \gamma^* \lambda - H(\tau, \gamma(\tau))\; dt \big] \\
		& = h_{1,0}(x) - h_{2,0}(x) =  S_{1}(q, \xi) - S_{2}(q,\eta)
	\end{align*}
	where $(q, \xi, \eta)$ is a critical point of $S_{1} \ominus S_{2}$. Hence, $c \left( \alpha, S_{1,t} \ominus S_{2,t} \right)$ is a critical values of  $S_{1} \ominus S_{2}$.
	\par Suppose at first that $S_{1} \ominus S_{2}$ has a finite number of critical values. Proposition \ref{Cont} shows that the map $t \mapsto c \left( \alpha, S_{1,t} \ominus S_{2,t} \right)$ is continuous. And since it takes its value in a discrete set, it must be constant.
	\par For the general case, knowing that the critical values of $S_{1} \ominus S_{2}$ are bounded, we can approximate $S_{1} \ominus S_{2}$ by a g.f.q.i (up to homotopy in the sense of Proposition \ref{FibredGFQI}) $S'_1 \ominus S'_2$ generating $ \mathcal{L}'_1 \string # \overline{\mathcal{L}'_2}$ and verifying the assumption above. We get that the map $t \mapsto c \left( \alpha, S'_{1,t} \ominus S'_{2,t} \right)$ is constant and by Proposition \ref{Cont}, it converges uniformly on compact sets to $t \mapsto c \left( \alpha, S_{1,t} \ominus S_{2,t} \right)$ as $S'_1 \ominus S'_2$ nears $S_{1} \ominus S_{2}$. Therefore, the latter map is also constant.
\end{proof}

\subsection{Graph Selectors} \label{GSsubsection}

We recall in this subsection the standard graph selector principal \cite{MR1094198,MR2025275,unknown}. As in these references, our approach is done in the g.f.q.i framework. However, an interested reader may refer to \cite{MR1484890} and \cite{MR3860396} for a more general approach based on Floer Homology.

\subsubsection{Generalities on Graph Selectors}

\begin{defi}
	Let $\mathcal{L}$ be an exact Lagrangian submanifold of $T^* \mathcal{M}$ generated by a g.f.q.i $S : \mathcal{M} \times \mathbb{R}^k \to \mathbb{R}$. For all $q \in \mathcal{M}$, denote by $\alpha_q$ the generator of $H^*(\{q\})$. The \textit{graph selector} of $\mathcal{L}$ associated to $S$ is the map $u_S : \mathcal{M} \to \mathbb{R}$ defined by
	\begin{equation}
		u_S(q) = c( \alpha_q , S_q) \quad \text{for } S_q := S_{|\{q\} \times \mathbb{R}^k}
	\end{equation}
\end{defi}

\begin{prop} \label{GSprop}
	The graph selector $u_S$ defined as above verify the following properties.
	\begin{enumerate}
		\item $u_S$ is Lipschitz on $\mathcal{M}$.
		\item There exists an open subset $U \subset \mathcal{M}$ of full Lebesgue measure such that $u_S$ is as regular as $S$ on $U$ and for all $q \in U$, 
		\begin{equation}
			(q, d_qu_S) \in \mathcal{L} \quad \text{and} \quad u_S(q) = h(q, d_q u_S) 
		\end{equation}
		with $h = S \circ i_S^{-1}$ the Liouville primitive on $\mathcal{L}$ relative to the g.f.q.i $S$.
	\end{enumerate} 
\end{prop}

\begin{proof}
	\begin{enumerate}
		\item Let $d$ be a Riemannian metric on $\mathcal{M}$. Since the map $(q,q',\xi) \in \mathcal{M}^2 \times \mathbb{R}^k \mapsto S(q,\xi) - S(q',\xi) \in \mathbb{R}$ is of compact support, it is Lipschitz and there exists a constant $K>0$ such that
		\begin{align*}
			\Vert S_{|\{q\} \times \mathbb{R}^k} - S_{|\{q'\} \times \mathbb{R}^k}\Vert _\infty & \leq K. d(q,q')
		\end{align*}
		Then, one can adapt the proof of Proposition \ref{Cont} and conclude that $u_S$ is $K$-Lipschitz.
		\item Let $U_0$ be the open set of regular values of the $C^1$ projection $\pi_S :\Sigma_S \to \mathcal{M}$. As $\dim \Sigma_S = \dim \mathcal{M}$, Sard's theorem can be applied to the projection map $\pi_S$ and we deduce that $U_0$ is of full measure. Let $q$ be a point of $U_0$. $\Sigma_S$ is transverse to $\{q\}\times \mathbb{R}^k$, then there exists a connected neighbourhood $V_q$ of $q$ in $\mathcal{M}$ such that $\pi_S^{-1}(V_q)$ is the disjoint union of $i$ graphs of $\xi_j : V_q \to \mathbb{R}^k$. Thus, $\mathcal{L}$ is the disjoint union of $i$ graphs of $ y \in V_z \mapsto ( y, d_qS(y, \xi_j(y)) \in T^* \mathcal{M}$ over $V_q$. Observe from the disjoint union that $d_qS(y,\xi_j(y))$ are pairwise distinct, then the sets $\{y \in V_q \; | \; S(y, \xi_{j_1}(y)) = S(y, \xi_{j_2}(y)) \}$ are discrete.
		
		Now set $U \subset U_0$ be the set of $q$ in $U_0$ such that the critical points of $S_q$ have pairwise distinct  critical values. The observation above shows that $U$ is a full measure open set of $\mathcal{M}$. Fix a point $q$ in $U$ and let $V_q$ be its corresponding neighbourhood defined as above. By Proposition \ref{Crit}, for all $y \in V_q$, $u_S(y)$ is one of the critical values of $S_y$ and more precisely one of the pairwise distinct values $S(y,\xi_j(y))$. And by continuity of the maps $y \mapsto u_S(y)$ and $y \mapsto S(y,\xi_j(y))$ and by connectedness of $V_q$, there exists a unique $j_q$ such that for all $y$ in $V_q$, $u_S(y) = S(y,\xi_{j_x}(y))$. Therefore, $u_S$ is as regular as $S$ at $q$ and $du_S(q) = d_qS(q,\xi_{j_q}(q))$. This equality yields $i_S(q,\xi_{j_q}(q))=(q, d_qu_S) \in \mathcal{L}$ and 
		\begin{align*}
			h(q, d_q u_S) = S \circ i_S^{-1} (q, d_qu_S) = S(q,\xi_{j_q}(q)) = u_S(q)
		\end{align*}
	\end{enumerate}
\end{proof}

\begin{prop} \label{GSbar}
	Let $\mathcal{L}$ be an exact Lagrangian submanifold of $T^* \mathcal{M}$ generated by a g.f.q.i $S : \mathcal{M} \times \mathbb{R}^k \to \mathbb{R}$. We have
	\begin{equation}
		u_{-S} = - u_S
	\end{equation}
\end{prop}

\begin{proof}
	The proof of $c( \alpha_q , -S_{|\{q\} \times \mathbb{R}^k}) = -c(\alpha_q, S_{|\{q\} \times \mathbb{R}^k})$ is a simpler counterpart of the proof of Proposition \ref{propc} restricted to the fibres.
\end{proof}

\begin{prop} \label{inegc}
	Let $\mathcal{L}$ be an exact Lagrangian submanifold of $T^* \mathcal{M}$ generated by a g.f.q.i $S : \mathcal{M} \times \mathbb{R}^k \to \mathbb{R}$. Let $1$ and $\mu$ be the respective generators of $H^0( \mathcal{M})$ and $H^d(\mathcal{M})$. Then we have the bounds
	\begin{equation}
		c(1,S) \leq u_S \leq c(\mu, S)
	\end{equation}
\end{prop}

\begin{proof}
	For $q \in \mathcal{M}$, set $S_q := S_{|\{q\} \times \mathbb{R}^k}$. Let $a$ be a real number. From the isomorphism depicted in (\ref{Thom}), we have the two identifications $H^*(\mathcal{M}) \simeq H^*(S^\infty, S^{-\infty})$ and $H^*(\{q\}) \simeq H^*(S_q^\infty, S_q^{-\infty})$. Additionally, the inclusion $i_q :\{q\} \hookrightarrow \mathcal{M}$ yields the morphisms $i_q^* : H^*(S^\infty, S^{-\infty}) \to H^*(S_q^\infty, S_q^{-\infty})$ and $i_{q,a}^* : H^*(S^a, S^{-\infty}) \to H^*(S_q^a, S_q^{-\infty})$ that send $1$ to $\alpha_q$.\\
	Thus, the inclusions $S_q^a \hookrightarrow S^a$ and $S^a \hookrightarrow S^\infty$ complete the following commutative diagram
	\begin{equation}
		\begin{tikzcd}
			H^*(\mathcal{M}) \simeq H^*(S^\infty, S^{-\infty}) \arrow{d} \arrow{r}{i_q^*}  &   H^*(\{q\}) \simeq H^*(S_q^\infty, S_q^{-\infty})  \arrow{d} \\
			H^*(S^a, S^{-\infty})  \arrow{r}{i_{q,a}^*} & H^*(S_q^a, S_q^{-\infty}) 
		\end{tikzcd}	
	\end{equation}		
	A diagram chasing provides the first inequality. If $a > c( \alpha_q , S_{|\{q\} \times \mathbb{R}^k})$, by definition of the spectral invariants $i_{a,q}^*(1) = \alpha_q \neq 0$ in $H^*(S_q^a, S_q^{-\infty})$. Hence, $1 \neq 0$ in $H^*(S^a, S^{-\infty})$ and we get the inequality $u_S(q) = c( \alpha_q , S_q) \geq c(1, S)$.
	
	We prove the second inequality using the Propositions \ref{propc} and \ref{GSbar}. Applying the first part of the proof to $\overline{\mathcal{L}}$, we get
	\begin{align*}
		c( \alpha_q , -S_q) \geq c(1, -S) = -c(\mu , S)
	\end{align*}
	and 
	\begin{align*}
		u_S(q) = - u_{-S}(q) = - c( \alpha_q , -S_q) \leq c(\mu , S)
	\end{align*}
\end{proof}

Recall that the \textit{oscillation} of a scalar map $f : X \to \mathbb{R}$ defined on a compact set $X$ is the quantity
\begin{equation}
	Osc (f) = \max f - \min f
\end{equation}
	
\begin{cor}
	With the notations of Proposition \ref{inegc}, if we consider $h$ the Liouville primitive on $\mathcal{L}$ relative to $S$, then we have
	\begin{equation}
		\min h \leq u_S \leq \max h
	\end{equation}
	and in particular 
	\begin{equation}
		Osc(u_S) \leq Osc(h)
	\end{equation}
\end{cor}
Note that this inequality is valid for all Liouville primitives on $\mathcal{L}$.

\begin{proof}
	This is an immediate consequence of Propositions \ref{inegc} and \ref{inegh}.
\end{proof}

\begin{prop} \label{GSTriangle}
	Let $S_1$ and $S_2$ be two g.f.q.i and consider $S = S_1 \oplus S_2$. Fix a point $q$ in $\mathcal{M}$ and let $\alpha_{1,q}$ and $\alpha_{2,q}$ be the generators of $H^*(S_{1,q}^\infty,S_{1,q}^{-\infty})$ and $H^*(S_{2,q}^\infty,S_{2,q}^{-\infty})$. Then $\alpha_q = \alpha_{1,q} \otimes \alpha_{2,q}$ generates $H^*(S^\infty, S^{-\infty})$ and we have the inequality
	\begin{equation} \label{GSTriangFormula}
		c(\alpha_{1,q},S_{1,q}) + c(\alpha_{2,q},S_{2,q}) = c(\alpha_{1,q} \otimes \alpha_{2,q},S_{1,q} \oplus S_{2,q}) = c(\alpha_q, S_q)
	\end{equation}
\end{prop}

\begin{proof}
	Let $m_1$, $m_2$ and $m=m_1+m_2$ be the respective indices of $S_{1,q}$, $S_{2,q}$ and $S_q$. We have $H^*(S_{1,q}^\infty,S_{1,q}^{-\infty}) \simeq H^*(\mathbb{D}^{m_1}, \partial \mathbb{D}^{m_1})$, $H^*(S_{2,q}^\infty,S_{2,q}^{-\infty}) \simeq H^*(\mathbb{D}^{m_2}, \partial \mathbb{D}^{m_2})$ and $H^*(S_{q}^\infty,S_{q}^{-\infty}) \simeq H^*(\mathbb{D}^{m}, \partial \mathbb{D}^{m})$ which yields the isomorphism
	\begin{equation*}
		\begin{split}
			H^*(S_{q}^\infty,S_{q}^{-\infty}) &\overset{\sim}{\longrightarrow} H^*(S_{1,q}^\infty,S_{1,q}^{-\infty}) \otimes H^*(S_{2,q}^\infty,S_{2,q}^{-\infty}) \\
			\alpha_q &\longmapsto \alpha_{1,q} \otimes \alpha_{2,q}
		\end{split}
	\end{equation*}
	Fix two real numbers $a$ and $b$. Then, we have the inclusions $S_{1,q}^a \times S_{2,q}^b \hookrightarrow S_q^{a+b}$ and for large $N$ and larger $N'>0$, we have inclusion $(S_{1,q}^a \times S_{2,q}^{-N'}) \cup (S_{1,q}^{-N'} \times S_{2,q}^b) \hookrightarrow S_q^{-N}$. These and the K\"unneth formila yield the morphism
	\begin{equation*}
		H^*(S_{q}^{a+b},S_{q}^{-\infty}) \longrightarrow H^*\left( S_{1,q}^a \times S_{2,q}^b, (S_{1,q}^a \times S_{2,q}^{-\infty}) \cup (S_{1,q}^{-\infty} \times S_{2,q}^b) \right) = H^*(S_{1,q}^a,S_{1,q}^{-\infty}) \otimes H^*(S_{2,q}^b,S_{2,q}^{-\infty})  
	\end{equation*}
	Hence, the inclusions $S_{1,q}^a \hookrightarrow S_{1,q}^\infty$, $S_{2,q}^a \hookrightarrow S_{2,q}^\infty$ and $S_{q}^{a+b} \hookrightarrow S_{q}^\infty$ complete the following commutative diagram
	\begin{equation}
		\begin{tikzcd}
			H^*(S_{q}^\infty,S_{q}^{-\infty})  \arrow{d} \arrow[r,"\sim"]  &  H^*(S_{1,q}^\infty,S_{1,q}^{-\infty}) \otimes H^*(S_{2,q}^\infty,S_{2,q}^{-\infty}) \arrow{d} \\
			H^*(S_{q}^{a+b},S_{q}^{-\infty}) \arrow{r} & H^*(S_{1,q}^a,S_{1,q}^{-\infty}) \otimes H^*(S_{2,q}^b,S_{2,q}^{-\infty}) \\
		\end{tikzcd}	
	\end{equation}	
	
	Let $a > c(\alpha_{1,q},S_{1,q})$ and $b > c(\alpha_{2,q},S_{2,q})$. We have $\alpha_{1,q} \neq 0$ and $\alpha_{2,q} \neq 0$ respectively in $H^*(S_{1,q}^\infty,S_{1,q}^{-\infty})$ and $H^*(S_{2,q}^\infty,S_{2,q}^{-\infty})$. Then by the commutativity of the diagram, the lower arrow sends $\alpha_q$ to $\alpha_{1,q} \otimes \alpha_{2,q} \neq 0$ and we deduce that $\alpha_q \neq 0$ in $H^*(S_{q}^{a+b},S_{q}^{-\infty})$. Hence, $a+b \geq  c(\alpha_q, S_q)$ and we deduce the inequality
	\begin{equation*}
		c(\alpha_{1,q},S_{1,q}) + c(\alpha_{2,q},S_{2,q}) \geq c(\alpha_q, S_q)
	\end{equation*}
	
	Now, let $c<c(\alpha_{1,q},S_{1,q}) + c(\alpha_{2,q},S_{2,q})$. If $c = a+b$, then $a < c(\alpha_{1,q},S_{1,q})$ or $b < c(\alpha_{2,q},S_{2,q})$ and in particular $\alpha_{1,q} \otimes \alpha_{2,q}$ vanishes on $S_{1,q}^{a} \times S_{2,q}^{b}$. Hence, $\alpha_q = \alpha_{1,q} \otimes \alpha_{2,q}$ vanishes on $S_q^c = \bigcup_{t\in \mathbb{R}} S_{1,q}^{c-t} \times S_{2,q}^t$ and we deduce that $c \leq c(\alpha_q,S_q)$. We obtain the inverse inequality
	\begin{equation*}
		c(\alpha_{1,q},S_{1,q}) + c(\alpha_{2,q},S_{2,q}) \leq  c(\alpha_q,S_q)
	\end{equation*}
\end{proof}

\begin{prop} \label{Bound}
	Let $\mathcal{L}_1$ and $\mathcal{L}_2$ be two Lagrangian submanifolds of $T^*\mathcal{M}$ respectively generated by the g.f.q.i $S_1$ and $S_2$. Let $u_{S_1}$ and $u_{S_2}$ be the corresponding graphs selectors. And let $1$ and $\mu$ be the respective generators of $H^0(\mathcal{M})$ and $H^d(\mathcal{M})$. Then we have the following bounds
	\begin{equation}
		c( 1, S_1 \ominus S_2 ) \leq u_{S_1} - u_{S_2} \leq c( \mu, S_1 \ominus S_2 )
	\end{equation}
\end{prop}

\begin{proof} 
	For $S := S_1 \ominus S_2$, an application of Propositions \ref{GSbar} and \ref{GSTriangle} yields
	\begin{align*}
		u_{S_1}(q) - u_{S_2}(q) = u_{S_1}(q) + u_{-S_2}(q) = c(\alpha_q, S_q)
	\end{align*}
	and we deduce from Proposition \ref{inegc} that
	\begin{align*}
		c(1, S) \leq u_{S_1}(q) - u_{S_2}(q) \leq c(\mu, S )
	\end{align*}
\end{proof}

\subsubsection{Construction of the Graph Selector of $\mathscr{L}$} \label{GSConstruction}

Let us summarize the constructions till this point of the article. We constructed $\mathpzc{h}$ a Liouville primitive on the extended non compact exact Lagrangian submanifold $\mathscr{L}$ defined in Section \ref{Extsection}. We had a path $(h_t)_{t \in \mathbb{R}}$ of Liouville primitives on $\mathcal{L}_t$ with the property that for all $t \in \mathbb{R}$ and $(q,p) \in \mathcal{L}_t$,
\begin{equation*}
	\mathpzc{h}(t,E,q,p) = h_t(q,p)
\end{equation*}
where $E$ is such that $(t,E,q,p) \in \mathscr{L}$. And for all times $T>0$, we set the paths $(S^T_t : M \times \mathbb{R}^{k_T} \to \mathbb{R})_{t \in [-T,T]}$ of g.f.q.i such that $S^T_t$ generates $\mathcal{L}_t= \phi_H^t( \mathcal{L})$ and
\begin{equation*}
	h_t = S^T_t \circ i_{S|\mathcal{L}_t}^{-1}
\end{equation*}

Now, for all $T>0$ we define the graph selector of $\mathcal{L}_t$ as
\begin{equation}
	u^T(t,q) = c(\alpha_q, S^T_{t|\{q\} \times \mathbb{R}^k})
\end{equation}

\begin{prop}
	For all $t \in [-T,T]$, $u^T_t(q) := u^T(t,\cdot)$ is independent of $T$.
\end{prop}

\begin{proof}
	Fix $0<T<T'$ and $t \in [-T,T]$. The g.f.q.i $S^T_t$ and $S^{T'}_t$ generate $\mathcal{L}_t$, then by the uniqueness Theorem \ref{Uniq} they are equivalent. These two generating functions have been constructed according to (\ref{ntransl}) that is $S^T_t \circ i_{S^T_t}^{-1} = S^{T'}_t \circ i_{S^{T'}_t}^{-1}$. This means that the translation never occurs in the successive operations intervening in the equivalence between $S^T_t$ and $S^{T'}_t$. Thus, we only need to take care of isomorphism and stabilization operations.
	
	We first deal with the isomorphism operation. Suppose that there exists a bundle isomorphism $\psi : M \times \mathbb{R}^{k_{T'}} \to M \times \mathbb{R}^{k_{T}}$ such that $S^{T'}_t = S^T_t \circ \psi$. Thus, for all $a \in \mathbb{R}$ and $q \in M$, $(S^{T'}_t)_q^a \simeq (S^T_t)_q^a$ and consequently $c(\alpha_q , (S^{T'}_t)_q) = c(\alpha_q , (S^{T}_t)_q)$ i.e $u^{T'}_t(q) = u^T_t(q)$.
	
	We deal now with the stabilization operation. Suppose that $k_{T'} = k_{T} + k$ and that there exists a non-degenerate quadratic form $Q : M \times \mathbb{R}^k \to \mathbb{R}$ such that $S^{T'}_t = S^T_t \oplus Q$. The quadratic form $Q_q=Q_{|{q}\times \mathbb{R}^k}$ is non-degenerate so that its only critical value in the fibre of $q$ is $c(\alpha^Q_q, Q_q) = 0$. Hence, we deduce from Proposition \ref{GSTriangle} that $u^{T'}_t(q) = u^T_t(q) + c(\alpha^Q_q, Q_q) = u^T_t(q)$.
\end{proof}

This independence enables us to define the map $u=u_\mathscr{L} : \mathcal{M} = \mathbb{R}\times M \to \mathbb{R}$ given by
\begin{equation}
	u_\mathscr{L}(t,q)= u^T_t(q) \quad \text{for some } T> |t|
\end{equation} 

\begin{prop} \label{Graphext}
	$u_\mathscr{L}$ is a graph selector for the exact Lagrangian submanifold $\mathscr{L}$ in the sense that it verifies the properties of Proposition \ref{GSprop}. More precisely
	\begin{enumerate}
		\item $u_\mathscr{L}$ is locally Lipschitz.
		\item There exists an open subset $\mathcal{U} \subset \mathcal{M}$ of full Lebesgue measure such that $u_\mathscr{L}$ is as regular as $\mathpzc{h}$ on $\mathcal{U}$ and for all $(t,q) \in \mathcal{U}$,
		\begin{equation}
			\big(t,\partial_tu(t,q),q, d_qu(t,q)\big) \in \mathscr{L} \quad \text{and} \quad u(t,q) = \mathpzc{h}\big(t,\partial_tu(t,q),q, d_qu(t,q)\big)
		\end{equation}
		with $\mathpzc{h}$ the Liouville primitive on $\mathscr{L}$ defined in Subsection \ref{LPrim}.
	\end{enumerate}
\end{prop}

\begin{proof}
	\begin{enumerate}
		\item Let $d$ be a Riemannian metric on $\mathcal{M} = \mathbb{R} \times M$. Let $(t,q)$ be a point of $\mathcal{M}$ and $T>|t|$ be a fixed time. Since the map $(t,q,t',q',\xi) \in \big([-T,T]\times M\big)^2 \times \mathbb{R}^k \mapsto S(t,q,\xi) - S(t',q',\xi) \in \mathbb{R}$ is of compact support, it is Lipschitz and there exists a constant $K_T>0$ such that
		\begin{align*}
			\Vert S_{|\{(t,q)\} \times \mathbb{R}^k} - S_{|\{(t',q')\} \times \mathbb{R}^k}\Vert _\infty & \leq K_T. d\big((t,q),(t',q')\big)
		\end{align*}
		Then, one can apply Proposition \ref{Cont} and conclude that $u_\mathscr{L}$ is $K_T$-Lipschitz on $[-T,T]\times M$. This being true for all $T>0$, we conclude that $u_\mathscr{L}$ is locally Lipschitz on $\mathcal{M}$.		
		\item Observe that for $T>0$ and $(t,q) \in (-T,T) \times M$, we have the equality $c(\alpha_{q}, S^T_{t|\{q\} \times \mathbb{R}^k}) = c(\alpha_{(t,q)}, S^T_{|\{(t,q)\} \times \mathbb{R}^k})$. Then, the proof of the second statement of Proposition \ref{GSprop} applied to $S^T$ instead of $S^T_t$ gives an open set $\mathcal{U}^T \subset (-T,T) \times M$ of full measure such that every element $(t,q)$ of $\mathcal{U}^T$ has a neighbourhood $\mathcal{V}_{(t,q)}$ and a regular map $\xi : \mathcal{V}_{(t,q)} \to \mathbb{R}^k$ such that for all $(s,y) \in \mathcal{V}_{(t,q)}$, $u_\mathscr{L}(s,y) = S^T\big(s,y,\xi(s,y)\big)$ which is as regular as $\mathpzc{h}$. When differentiated, this gives
	\begin{equation}
		\partial_t u_\mathscr{L}(s,y) = \partial_t S^T\big(s,y,\xi(s,y)\big) \quad \text{ and } \quad d_q u_\mathscr{L}(s,y) = d_q S^T\big(s,y,\xi(s,y)\big)
	\end{equation}
	Therefore, Proposition \ref{GFext} imply that $\big(t,\partial_tu(t,q),q, d_qu(t,q)\big)$ belongs to $\mathscr{L}$ meaning that
	\begin{align*}
		u(t,q) = u^T_t(q) = h_t(q,d_q u^T(q)) = \mathpzc{h}\big(t,\partial_tu(t,q),q, d_q u(t,q)\big)
	\end{align*}
	Finally, the set $\mathcal{U} = \bigcup_{T \in \mathbb{N}^*} \mathcal{U}^T$ of $\mathcal{M}$ is open of full measure and meets the required properties.
	\end{enumerate}
\end{proof}

\section{The Birkhoff Theorem for the Lagrangian Submanifold $\mathscr{L}$} \label{Proofsection}

The current section is devoted to the proof of Theorem \ref{MainTheorem}.  We start in Subsection \ref{Calibsection} by presenting some concepts that emerged in Fathi's weak-KAM theory that will show crucial in showing the main result. See \cite{fathi2008weak} \cite{MR1720372} for elaborate expositions on the subject in the autonomous framework. The remainder of the section is full focused on the proof.

\subsection{Calibration} \label{Calibsection} 

The Hamiltonian $H : \mathbb{T}^1 \times T^*M \to \mathbb{R}$ is assumed to be Tonelli. This enables to define its convex conjugate named the \textit{Lagrangian} $L: \mathbb{T}^1 \times TM \to \mathbb{R}$ given by the following formula
\begin{equation} \label{Lag}
	L(t,q,v) = \max_{p \in T^*_qM} \{ p(v) - H(t,q,p) \} 
\end{equation}

We introduce a tool derived from convex analysis that will help in the upcoming proofs.
\begin{lem}(Fenchel's inequality)
	For all $q$ in $M$ and all $(v,p) \in T_qM \times T^*_qM$
	\begin{equation} \label{Fenchel}
		p(v) \leq H(t,q,p) + L(t,q,v)
	\end{equation}
	with equality if and only if $p = \partial_v L (q,v)$ if and only if $v= \partial_p H(t,q,p)$.
\end{lem}

\begin{rem}
	Note that for $x=(q,p)$, if $\phi^t_H(x)=(q(t),p(t)) $ is a curve that follows the Hamiltonian flow, then the Hamiltonian equations results in the equalities
	\begin{equation} \label{LegendreMapsEL}
		\dot{q}(t) = \partial_p H\big(t,q(t),p(t)\big) \quad \text{and} \quad p(t) = \partial_v L\big(t,q(t),\dot{q}(t) \big)
	\end{equation}
\end{rem}

Another fundamental concept that has emerged from weak-KAM theory is the following.

\begin{defi}
	Let $\gamma : [a,b] \to M$ be a $C^1$ curve on $M$. The \textit{defect of calibration} of $\gamma$ is defined as
	\begin{equation}
		\delta(u,\gamma) = \int_a^b L(s,\gamma(s),\dot{\gamma}(s)) \; ds - [u(b,\gamma(b)) - u(a,\gamma(a))]
	\end{equation}
\end{defi}

Recall that $u$ is the graph selector of the extended Lagrangian submanifold $\mathscr{L}$ constructed in Subsection \ref{GSConstruction}.

\begin{prop} \label{Domination}
	The defect of calibration $\delta$ is always non-negative. We say that the map $u : \mathbb{R} \times M \to \mathbb{R}$ is \textit{dominated} by $L$.
\end{prop}

\begin{proof}
	Let $\gamma : [a,b] \to M$ be a $C^1$ curve on $M$ such that for almost every time $t \in [a,b]$, $(t,\gamma(t))$ belongs to the dense open set $\mathcal{U}$ defined in Proposition \ref{Graphext}.
	\begin{align*}
		u(b,\gamma(b)) - u(a,\gamma(a)) & = \int_a^b du(s,\gamma(s)).(1,\dot{\gamma}(s)) \;ds = \int_a^b \partial_t u(s,\gamma(s)) + d_q u (s,\gamma(s)).\dot{\gamma}(s) \;ds \\
		& \leq \int_a^b \partial_t u(s,\gamma(s)) + H\big(s,\gamma(s),d_q u (s,\gamma(s))\big) + L\big(s,\gamma(s),\dot{\gamma}(s)\big) \;ds \\
		& = \int_a^b \mathscr{H} \big(s, \partial_t u(s,\gamma(s)),\gamma(s),d_q u (s,\gamma(s))\big) \; ds + \int_a^b L(s,\gamma(s),\dot{\gamma}(s)) \; ds 
	\end{align*}
	where in the second line we used the Fenchel inequality (\ref{Fenchel}).
	
	Now by Proposition \ref{Graphext}, we have from the assumption on $\gamma$ that for almost all $s \in [a,b]$, 
	\begin{equation}
		\big(s, \partial_t u(s,\gamma(s)),\gamma(s),d_q u (s,\gamma(s))\big) \in \mathscr{L} \subset \{ \mathscr{H} = 0 \}
	\end{equation} 
	Then we get
	\begin{equation*}
		\int_a^b \mathscr{H} \big(s, \partial_t u(s,\gamma(s)),\gamma(s),d_q u (s,\gamma(s))\big) \; ds = 0
	\end{equation*}
	and
	\begin{equation*}
		\delta(u,\gamma) = \int_a^b L(s,\gamma(s),\dot{\gamma}(s)) \; ds - [u(b,\gamma(b)) - u(a,\gamma(a))] \geq 0
	\end{equation*}
	The general configuration case is given by the perturbation process described in Lemma \ref{Pertrub}.
\end{proof}

\begin{lem} \label{Pertrub}
	Let $\gamma : [a,b] \to M$ be a $C^1$ curve and let $\mathcal{U}$ be a full Lebesgue measure set of $\mathcal{M} = \mathbb{R} \times M$. Then there exists a sequence of $C^1$ curves $\gamma_k : [a,b] \to M$ such that
	\begin{enumerate}[label=\roman*.]
		\item For almost all $t \in [a,b]$, $(t, \gamma_k(t)) \in \mathcal{U}$.
		\item The sequence $(\gamma_k)_k$ converges to $\gamma$ in the $C^1$-topology.
	\end{enumerate}
\end{lem}

\begin{proof}
	$\gamma : [a,b] \to M$ is a $C^1$ curve in $M$, then we can extend it to a $C^1$ curve defined on $(c,d) \supset [a,b]$. The curve $(t,\gamma(t))$ is $C^1$ embedded in $\mathcal{M} = \mathbb{R} \times M$, then it admits a tubular neighbourhood $\mathcal{O}$ and there exists a $C^1$ embedding $\psi : \mathcal{O} \to \mathbb{R} \times \mathbb{R}^n$ such that for all $t \in (c,d)$, $\psi(t, \gamma(t)) = (t,0)$.
	
	Take $\mathcal{R}' = [a,b] \times [-\varepsilon,\varepsilon]^n \subset \psi(\mathcal{O})$ and $\mathcal{R} = \psi^{-1}(\mathcal{R}')$. Denote by $l$ the Lebesgue measure on $\mathcal{M}$ and $l' = \psi_*(l)$ be a measure on $\psi( \mathcal{O})$. Since $\mathcal{U}_0 := \mathcal{U} \cap \mathcal{R}$ is of full measure in $\mathcal{R}$, $\mathcal{U}'_0 := \psi( \mathcal{U}_0)$ is of full measure in $\mathcal{R}'$. Then using Fubini
	\begin{align*}
		l'(\mathcal{R'}) = l'(\mathcal{U}'_0) = \int_{[-\varepsilon,\varepsilon]^n} \int_{([a,b] \times \{q'\})  \cap \mathcal{U}_0' } dt'\; dq' \leq \int_{[-\varepsilon,\varepsilon]^n} \int_{([a,b] \times \{q'\}) } dt'\; dq' = l'( \mathcal{R}')
	\end{align*}
	Hence, by positivity of the integrands, we get that for almost every $q'$ in $[-\varepsilon, \varepsilon ]^n$, 
	\begin{equation} \label{almostx}
		\int_{([a,b] \times \{q'\})  \cap \mathcal{U}_0' } dt' = \int_{([a,b] \times \{q'\}) } dt'
	\end{equation}   
	or in other words, for all such $q'$, for almost every $t'$ in $[a,b]$, $(t',q')$ belongs to $\mathcal{U}_0$.
	
	Let $(q'_k)_k$ be a sequence in $[-\varepsilon, \varepsilon]^n$ such that every element $q_k$ verifies the equality (\ref{almostx}) and the sequence converges to $0$ as $k$ goes to infinity. For all $k$, we define the constant curves $\gamma'_k \equiv q'_k: [a,b] \to \mathcal{R}'$ and $\gamma_k(t) = \psi^{-1}(t,\gamma'_k)$. The curves $\gamma_k'$ converge to $\psi(\gamma)$ in the $C^1$-topology and since $\psi$ is regular, the same holds for the curves $\gamma_k$. Moreover, from (\ref{almostx}) we know that for almost all $t$ in $[a,b]$, $(t, \gamma_k(t))$ belongs to $\mathcal{U}$ which concludes the proof.
\end{proof}

Now that the domination of $u$ has been established, we can take more interest in curves that achieve equality in this domination, or in other words curves having null calibration defect.

\begin{defi}
	A $C^1$ curve $\gamma : I \to M$ on $M$ is said \textit{calibrated} by $u$ if for all $[a,b] \subset I$, $\delta(u,\gamma_{|[a,b]}) = 0$.
\end{defi}

\begin{rem}
	We know from Proposition \ref{Domination} that for all times $t<t'$ and points $q$ and $q'$ of $M$,
	\begin{equation}
		u(t',q') - u(t,q) \leq \inf \left\{ \int_t^{t'} L(\tau,\gamma(\tau),\dot{\gamma}(\tau)) \; d\tau \; \left| \;
		\begin{matrix}
			\gamma : & [t,t'] \to M \\
			& s \mapsto q \\
			& t \mapsto q'
		\end{matrix} \right.	\right\}
	\end{equation}
	Hence, if $\gamma : [a,b] \to M$ is $u$-calibrated, then it is minimizing for the Lagrangian $L$.
\end{rem}

\begin{prop} \label{Fathi}
	\begin{enumerate}
		\item If $\gamma : I \to M$ is a $u$-calibrated curve, then for all time $t$ in the interior of $I$, $u$ is differentiable at $(t, \gamma(t))$ and
		\begin{equation}
			d_qu(t,\gamma(t)) = \partial_vL \big(t,\gamma(t), \dot{\gamma}(t) \big) \quad \text{and} \quad \mathscr{H}\big(t,\partial_tu(t,\gamma(t)), \gamma(t), d_qu(t,\gamma(t)) \big) =0
		\end{equation}

		\item Let $A_{\varepsilon,u}$ be the set of points $(t,q) \in \mathbb{R} \times M$ such that there exists a curve $\gamma_{t,q} : (t-\varepsilon, t+\varepsilon) \to M$ with $\gamma_{t,q}(t) = q$ which is $u$-calibrated. Then, the map 
		\begin{align*}
			A_{\varepsilon,u} &\longrightarrow T^*(\mathbb{R} \times M) \\
			(t,q) &\longmapsto du (t,q) = \big(t,\partial_tu(t,q), q, d_qu(t,q)\big)
		\end{align*}
		is locally Lipschitz. 
	\end{enumerate}		
\end{prop}

\begin{proof}
	\textit{1. Differentiability.} Fix $(t,q) = (t,\gamma(t))\in \mathbb{R} \times M$. We will bound $u$ in a neighbourhood of $(t,q)$ by two $C^1$ maps that coincide with it at this point. In order to do so, we use the domination inequality. Let $(s,y)$ be close enough to $(t,q)$ and fix two reference times $t^+<t<t^-$ and $q^\pm = \gamma(t^\pm)$. From the calibration
	\begin{equation} \label{Calibdiff}
		u(t,q) = u(t^\pm,q^\pm) + \int_{t^\pm}^t L\big(\tau,\gamma(\tau),\dot{\gamma}(\tau) \big)\; d\tau =: \psi^\pm(t,q)
	\end{equation}
	and from the domination
	\begin{equation} \label{psi+formula}
		u(s,y) \leq u(t^+,q^+) + \int_{t^+}^t L\big(\tau,\gamma^+_{(s,y)}(\tau),\dot{\gamma}^+_{(s,y)}(\tau) \big)\; d\tau =: \psi^+(s,y)
	\end{equation}
	and 
	\begin{equation}
		u(s,y) \geq u(t^-,q^-) - \int_t^{t^-} L\big(\tau,\gamma^-_{(s,y)}(\tau),\dot{\gamma}^-_{(s,y)}(\tau) \big)\; d\tau =: \psi^-(s,y)
	\end{equation}
	where, in a chart around $(t,q)$
	\begin{equation}
		\gamma^\pm_{(s,y)}(\tau) = \gamma(\tau) + \frac{\tau - t^\pm}{s - t^\pm} (y-\gamma(s))
	\end{equation}		
	are smooth families of curves linking $(t^\pm,q^\pm)$ to $(s,y)$ and such that $\gamma^\pm_{(t,q)}= \gamma$.\\
	It is easy to see that $\psi^\pm$ are $C^1$. Moreover, $\psi^- \leq u \leq \psi^+$ with equalities at $(t,q)$. Then $u$ is $C^1$ at $(t,q)$.\\
	
	\textit{Evaluation of the Differential.} We differentiate (\ref{Calibdiff}) with respect to time $t$ without forgetting that $q = \gamma(t)$, to get
	\begin{equation}
		\partial_tu(t,\gamma(t)) + d_q u(t,\gamma(t)) = L\big(t,\gamma(t), \dot{\gamma}(t) \big)
	\end{equation}
	And by Fenchel inequality (\ref{Fenchel}) for $q = \gamma(t)$, $v = \dot{\gamma}(t)$ and $p = d_q u(t,\gamma(t))$, we have
	\begin{multline} \label{CalibrationDem0}
		0 = \partial_tu(t,\gamma(t)) + d_q u(t,\gamma(t)). \dot{\gamma}(t) - L\big(t,\gamma(t), \dot{\gamma}(t) \big) \\
		\leq \partial_tu(t,q) + H(t,q,d_q u(t,q)) = \mathscr{H}\big(t,\partial_tu(t,\gamma(t)), \gamma(t), d_qu(t,\gamma(t)) \big)
	\end{multline}

	In order to obtain the equality, we need to show that $\mathscr{H}\big(t,\partial_tu(t,\gamma(t)), \gamma(t), d_qu(t,\gamma(t)) \big) \leq 0$. This will follow from the convexity of the Hamiltonian $H$ on fibres and from a result due to Clarke (see \cite{MR2346451} for a proof or \cite{MR0709590} for a more general result).
	\begin{lem} \label{ConvexHull}
		Let $f : U \to \mathbb{R}$ be a Lipschitz map defined on an open subset $U$ of $\mathbb{R}^d$ and let $U_0 \subset U$ be a subset of full Lebesgue measure. Let $q$ be a fixed element of $U$. We introduce the following sets
		\begin{equation}
			\begin{split}
				K^{U_0}_f(q) &:= \big\{ \text{limit points of } \big(df(q_n)\big)_{n \geq 0} \;  \big| \; q_n \in U_0, \; \lim\limits_n q_n = q  \big\} \subset T^*_qU \\
				C^{U_0}_f(q) &:= \conv (K^{U_0}_f(q))
			\end{split}
		\end{equation}
		where $\conv$ stands for the convex hull. Then, whenever $f$ is differentiable at a point $q \in U$, we have $df(q) \in C^{U_0}_f(q)$.
	\end{lem}	 
	We apply the lemma to the map $u : \mathbb{R} \times M \to \mathbb{R}$ and to the full measure subset $\mathcal{U}$ of $\mathbb{R} \times M$ introduced in Proposition \ref{Graphext}. We get that 
	\begin{equation} \label{CalibratinoDem1}
		du(t,\gamma(t)) = \big(\partial_tu(t,\gamma(t)), d_qu(t,\gamma(t))\big) \in C^{\mathcal{U}}_u(t,\gamma(t))
	\end{equation}

	Additionally, we know from the definition of the set $\mathcal{U}$ that for all $(s,q) \in \mathcal{U}$, $\big(s,\partial_tu(s,q), q, d_qu(s,q) \big)  \subset \mathscr{L} \subset \{\mathscr{H} =0 \}$. Hence, we deduce by continuity of $\mathscr{H}$ that 
	\begin{equation} \label{CalibrationDem2}
		K^\mathcal{U}_u(t,\gamma(t)) \subset \{\mathscr{H}=0\} \cap T^*_{(t,\gamma(t))}( \mathbb{R} \times M ) \quad \text{and} \quad C^\mathcal{U}_u(t,\gamma(t)) \subset \conv \{\mathscr{H}=0\} \cap T^*_{(t,\gamma(t))}( \mathbb{R} \times M )
	\end{equation}
	Moreover, 
	\begin{align*}
		\{\mathscr{H} \leq 0\} \cap T^*_{(t,\gamma(t))}( \mathbb{R} \times M ) = \{ (E,p) \in T^*_{(t,\gamma(t))}( \mathbb{R} \times M ) \; | \;  H(t,\gamma(t),p ) \leq -E \}
	\end{align*}
	which corresponds, up to the symmetry $E \mapsto -E$, to the epigraph of the strictly convex Hamiltonian $H$ restricted to the fibre $T^*_{(t,\gamma(t))}( \mathbb{R} \times M )$. Thus, this set is strictly convex with boundary (or extremal points)
	\begin{align*}
		\{ (E,p) \in T^*_{(t,\gamma(t))}( \mathbb{R} \times M ) \; | \;  H(t,\gamma(t),p ) = -E \} =  \{\mathscr{H} = 0\} \cap T^*_{(t,\gamma(t))}( \mathbb{R} \times M )
	\end{align*}
	Therefore, we deduce that $\conv \{\mathscr{H}=0\} \cap T^*_{(t,\gamma(t))}( \mathbb{R} \times M ) = \{\mathscr{H} \leq 0\} \cap T^*_{(t,\gamma(t))}( \mathbb{R} \times M )$. Regrouping the inclusions (\ref{CalibratinoDem1}) and (\ref{CalibrationDem2}), we obtain 
	\begin{equation*}
		du(t,\gamma(t)) \in \{\mathscr{H} \leq 0 \}  \cap T^*_{(t,\gamma(t))}( \mathbb{R} \times M ) \subset  \{\mathscr{H} \leq 0 \}
	\end{equation*}    
	and 
	\begin{equation} \label{CalibrationDem4}
		\mathscr{H} \big(t,\partial_tu(t,\gamma(t)), \gamma(t), d_qu(t,\gamma(t)) \big) \leq 0
	\end{equation}
	This implies the equality all along the inequalities of (\ref{CalibrationDem0}) and (\ref{CalibrationDem4}) leading to the desired identities.\\
	
	\textit{2. $C^{1,1}$ Regularity.} We aim to use Fathi's criterion for a Lipschitz derivative, the proof of which can be found in proposition 4.11.3 of \cite{fathi2008weak}.
	\begin{lem} \label{C11Criterion}
		Fix a point $q_0$ of $\mathbb{R}^d$ and a radius $r>0$. Let $u : B(q_0,r) \to \mathbb{R}$ be a function and let $C>0$ be a positive constant. We introduce the set $A^{C,u}$ of $B(q_0,r)$ as
		\begin{align*}
			A^{C,u} = \big\{q \in B(q_0,r) \; | \; \exists \varphi_q : \mathbb{R}^d \to \mathbb{R} \text{ linear, } \forall y\in B(q_0,r), \; \Vert u(y) - u(q) - \varphi_q(y-q) \Vert  \leq C\Vert y-q\Vert ^2 \big\}
		\end{align*}
		Then $u$ has for all $q \in A^{C,u}$, $d_qu = \varphi_q$ and the restriction of $q \mapsto d_qu$ to $\{ q \in A^{C,u} \; | \; \Vert q-q_0\Vert  \leq r /3\}$ is Lipschitz with Lipschitz constant $6C$.
	\end{lem}
	We will search for a constant $C$ such that the elements $A_{\varepsilon,u} \subset A^{C,u}$. Keeping the notation of the previous part of the proof, we set $t^\pm = t\pm \varepsilon$ and we consider a neighbourhood $[s^+,s^-] \times V$ of $(t,q)$ such that the curves $\gamma^\pm_{(s,y)}(\tau)$ remains in the chart we are working in for all $\tau \in [t^+,t^-]$. We take $t^+ + \varepsilon/2  < s^+ < t < s^- < t^- - \varepsilon/2$ and $V \subset B(q,\varepsilon)$.
	
	We have $\psi^- \leq u \leq \psi^+$ and $\psi^- (t,q) = u (t,q)= \psi^+(t,q)$. Thus, for $(s,y) \in (t^-, t^+) \times M$, 
	\begin{equation} \label{CalibRegDem0}
		\psi^-(s,y) - \psi^-(t,q) \leq u(s,y) - u(t,q) \leq \psi^+(s,y) - \psi^+(t,q)
	\end{equation}
	We estimate the right-hand side by applying a 2 Taylor expansion on the map $\psi^+$ at $(t,q)$.
	\begin{equation} \label{CalibRegDem11}
		\Big|  \psi^+(s,y)  - \psi^+(t,q) - d\psi^+(t,q).(s-t,y-q)  \Big| \leq \left\Vert d^2 \psi^+ \right\Vert_\infty^{[s^+, s^-] \times V} .\big( |s-t|^2 + \Vert y-q\Vert ^2 \big) 
	\end{equation}
	where $\Vert \cdot \Vert_\infty^{[s^+, s^-] \times V}$ stands for the $\Vert \cdot \Vert_\infty$-norm of the restriction to the set $[s^+, s^-] \times V$. We set 
	\begin{equation} \label{CalibRegDem12}
		C' = \left\Vert d^2 \psi^+ \right\Vert_\infty^{[s^+, s^-] \times V}
	\end{equation}
	We would like to see that it is independent of $(t,q)$. This is given by the following classical lemma first proved by John N. Mather in \cite{MR1109661}.
	
	\begin{lem}(A Priori Compactness) \label{APrioriCompactness}
		Let $L : TM \to \mathbb{R}$ be a Tonelli Lagrangian and fix a small positive $\varepsilon >0$. Then, there exists a compact subset $K_\varepsilon$ of $TM$ such that every minimizing curve $\gamma : [s,t] \to M$ with $t-s \geq \varepsilon$ verifies $(\gamma(\tau) , \dot{\gamma}(\tau)) \in K_\varepsilon$.
	\end{lem}
	Since the curve $\gamma$ is minimizing and of time length $2\varepsilon$, we infer from the à priori compactness that $(\gamma, \dot{\gamma})$ is contained in the compact set $K_{2\varepsilon}$. Moreover, since the minimizing curves follow the Lagrangian flow $\phi_L$ of $L$, we have $\ddot{\gamma}(\tau) = dv \circ X_L(\tau, \gamma(\tau),\dot{\gamma}(\tau))$ so that $\{\ddot{\gamma}(\tau) \; | \; \tau \in [t^+,t^-] \}$ is contained in a compact set that only depends on $\varepsilon$. Consequently, the set $\{(\gamma^\pm_{(s,y)}(\tau), \dot{\gamma}^\pm_{(s,y)}(\tau), \ddot{\gamma}^\pm_{(s,y)}(\tau) ) \; | \; (\tau,s,y) \in [t^+,t^-] \times [s^+, s^-] \times V \}$ is contained in a compact set $K$ of $T^*M$ independent of $(t,q)$. Moreover, we took $[t^+,t^-] \times [s^+,s^-] \times V \subset [t-\varepsilon, t+ \varepsilon] \times [t-\varepsilon/2, t+ \varepsilon/2] \times B(q,\varepsilon)$. This shows that the constant $C'$ only depends on $\varepsilon$.\\   
	
	Let us now evaluate $d\psi^+(s,y)$ where we recall the definition \eqref{psi+formula} of $\psi^+$. We have
	\begin{equation}
		\begin{split}
			\partial_s \psi^+(t,q) &= \int_{t^+}^t \partial_q L(\tau,\gamma(\tau), \dot{\gamma}(\tau)). \partial_s \gamma^+_{(t,q)}(\tau) \; d\tau \\
			&+ \int_{t^+}^t \partial_vL (\tau,\gamma(\tau), \dot{\gamma}(\tau)). \partial_s \dot{\gamma}^+_{(t,q)}(\tau) \; d\tau  + L(t,\gamma(t), \dot{\gamma}(t)) \\
			d_y \psi^+(t,q) &= \int_{t^+}^t \partial_q L(\tau,\gamma(\tau), \dot{\gamma}(\tau)).d_y \gamma^+_{(t,q)}(\tau) \; d\tau \\
			&+ \int_{t^+}^t \partial_vL (\tau,\gamma(\tau), \dot{\gamma}(\tau)).  d_y \dot{\gamma}^+_{(t,q)}(\tau) \; d\tau
		\end{split}
	\end{equation}
	so that
	\begin{equation} \label{CalibRegDem10}
		\begin{split}
			d\psi^+(t,q).(s-t,y-q) &=\partial_s \psi^+(t,q).(s-t) + d_y \psi^+(t,q).(y-q) \\
			&= \int_{t^+}^t \partial_q L(\tau,\gamma(\tau), \dot{\gamma}(\tau)).\big[ \partial_s \gamma^+_{(t,q)}(\tau). (s-t) + d_y \gamma^+_{(t,q)}(\tau). (y-q) \big] \; d\tau \\
			&+ \int_{t^+}^t \partial_vL (\tau,\gamma(\tau), \dot{\gamma}(\tau)). \big[ \partial_s \dot{\gamma}^+_{(t,q)}(\tau). (s-t) + d_y \dot{\gamma}^+_{(t,q)}(\tau). (y-q) \big] \; d\tau \\
			&+ L(t,\gamma(t), \dot{\gamma}(t)).(s-t) 
		\end{split}
	\end{equation}	
	
	We evaluate the differential of $\gamma^+_{(s,y)}$ and $\dot{\gamma}^+_{(s,y)}$ with respect to $s$ and $y$. Recall that
	\begin{align*}
		\gamma^+_{(s,y)}(\tau) = \gamma(\tau) + \frac{\tau - t^+}{s - t^+} (y-\gamma(s)) \quad \text{and} \quad \dot{\gamma}^+_{(s,y)}(\tau) = \dot{\gamma}(\tau) + \frac{1}{s - t^+} (y-\gamma(s))
	\end{align*}
	so that
	\begin{align*}
		\partial_s \gamma^+_{(s,y)}(\tau) &= - \frac{\tau - t^+}{(s-t^+)^2} (y- \gamma(s)) - \frac{\tau - t^+}{s- t^+} \dot{\gamma}(s), & \partial_s \gamma^+_{(t,q)}(\tau) &= - \frac{\tau - t^+}{t-t^+} \dot{\gamma}(t) \\
		d_y \gamma^+_{(s,y)}(\tau) &=  \frac{\tau - t^+}{s-t^+} dy, & d_y \gamma^+_{(t,q)}(\tau).(y-q) & = \frac{\tau - t^+}{t-t^+} (y-q)
	\end{align*}
	and
	\begin{align*}
		\partial_s \dot{\gamma}^+_{(s,y)}(\tau) &= - \frac{1}{(s-t^+)^2} (y- \gamma(s)) - \frac{1}{s- t^+} \dot{\gamma}(s), & \partial_s \dot{\gamma}^+_{(t,q)}(\tau) &= - \frac{1}{t-t^+} \dot{\gamma}(t) \\
		d_y \dot{\gamma}^+_{(s,y)}(\tau) &=  \frac{1}{s-t^+} dy, & d_y \dot{\gamma}^+_{(t,q)}(\tau).(y-q) & = \frac{1}{t-t^+} (y-q)
	\end{align*}
	Hence, we get
	\begin{multline*}
		\partial_q L(\tau,\gamma(\tau), \dot{\gamma}(\tau)).\big[ \partial_s \gamma^+_{(t,q)}(\tau). (s-t) + \partial_y \gamma^+_{(t,q)}(\tau). (y-q) \big]\\
		= \partial_q L(\tau,\gamma(\tau), \dot{\gamma}(\tau)). \big[ - \frac{\tau - t^+}{t-t^+}(s-t). \dot{\gamma}(t)  +   \frac{\tau - t^+}{t-t^+} (y-q)  \big] \\
		= (\tau - t^+). \partial_q L(\tau,\gamma(\tau), \dot{\gamma}(\tau)). \big[ - \frac{1}{t-t^+}(s-t). \dot{\gamma}(t)  +   \frac{1}{t-t^+} (y-q)  \big]
	\end{multline*}
	and
	\begin{multline*}
		\partial_vL (\tau,\gamma(\tau), \dot{\gamma}(\tau)). \big[ \partial_s \dot{\gamma}^+_{(t,q)}(\tau). (s-t) + \partial_y \dot{\gamma}^+_{(t,q)}(\tau). (y-q) \big] \\
		= \partial_vL (\tau,\gamma(\tau), \dot{\gamma}(\tau)). \big[ - \frac{1}{t-t^+}(s-t).\dot{\gamma}(t) + \frac{1}{t-t^+} (y-q) \big]
	\end{multline*}
	And since the curve $\gamma$ is calibrated, it is minimizing and it verifies the Euler-Lagrange equation (see proposition 2.2.6 of \cite{fathi2008weak})
	\begin{align*}
		(\tau - t^+). \partial_q L(\tau,\gamma(\tau), \dot{\gamma}(\tau)) + \partial_vL (\tau,\gamma(\tau), \dot{\gamma}(\tau)) &= (\tau - t^+). \frac{d}{d\tau} \big(  \partial_vL (\tau,\gamma(\tau), \dot{\gamma}(\tau)) \big) +  \partial_vL (\tau,\gamma(\tau), \dot{\gamma}(\tau))  \\
		&= \frac{d}{d\tau} \Big( (\tau - t^+).\partial_vL (\tau,\gamma(\tau), \dot{\gamma}(\tau)) \Big)
	\end{align*}
	Going  back to identity \eqref{CalibRegDem10}, we have shown that
	\begin{equation} \label{CalibRegDem13}
		\begin{split}
			d\psi^+(t,q).(s-t,y-q) &= (t - t^+).\partial_vL (t,\gamma(t), \dot{\gamma}(t)) . \big[ - \frac{1}{t-t^+}(s-t).\dot{\gamma}(t) + \frac{1}{t-t^+} (y-q) \big] \\
			&+ L(t,\gamma(t), \dot{\gamma}(t)).(s-t)  \\
			&= \big[ -\partial_vL (t,\gamma(t), \dot{\gamma}(t)).\dot{\gamma}(t) + L(t,\gamma(t), \dot{\gamma}(t)) \big].(s-t) + \partial_vL (t,\gamma(t), \dot{\gamma}(t)).(y-q) \\
			&= - H \big(t,q, \partial_vL (t,q, \dot{\gamma}(t)) \big) .(s-t) + \partial_vL (\tau,q, \dot{\gamma}(t)).(y-q)
		\end{split}
	\end{equation}
	where we last used the equality case of the Fenchel inequality (\ref{Fenchel}).
		
	Gathering \eqref{CalibRegDem11}, \eqref{CalibRegDem12} and \eqref{CalibRegDem13}, we obtain
	\begin{equation*}
		\psi^+(s,y)  - \psi^+(t,q) \leq - H \big(t,q, \partial_vL (t,q, \dot{\gamma}(t)) \big) .(s-t) + \partial_vL (\tau,q, \dot{\gamma}(t)).(y-q) + C' \big( |s-t|^2 + \Vert y-q\Vert ^2 \big) 
	\end{equation*}
	Analogously for $\psi^-$, we find a constant $C''>0$ depending only on $\varepsilon$ such that
	\begin{equation*}
		\psi^-(s,y)  - \psi^-(t,q) \geq - H \big(t,q, \partial_vL (t,q, \dot{\gamma}(t)) \big) .(s-t) + \partial_vL (\tau,q, \dot{\gamma}(t)).(y-q) - C'' \big( |s-t|^2 + \Vert y-q\Vert ^2 \big) 
	\end{equation*}
	And we finally get from (\ref{CalibRegDem0}) that
	\begin{align*}
		\big| u(s,y) - u(t,q) + H\big(t,q, \partial_vL (t,q, \dot{\gamma}(t)) \big).(s-t) - \partial_vL (\tau,q, \dot{\gamma}(t)).(y-q) \big| \leq C\big( |s-t|^2 + \Vert y-q\Vert ^2 \big)
	\end{align*}
	This allows to apply Lemma \ref{C11Criterion} and to conclude that 
	\begin{align*}
		d_qu(t,\gamma(t)) = \partial_vL \big(t,\gamma(t), \dot{\gamma}(t) \big), \quad \partial_tu(t,q) = -H (t,q,d_qu(t,q))
	\end{align*}
	and that the map $(t,q) \longmapsto du (t,q) = \big(t,\partial_tu(t,q), q, d_qu(t,q)\big)$ restricted to $A_{\varepsilon,u}$ is locally Lipschitz.
\end{proof}

\subsection{Setting and Notations} \label{SectionSetting}

We keep the same notations as in the statement of Theorem \ref{MainTheorem}.\\

Let $\mathcal{L}$ be a Lagrangian submanifold of $T^*M$, $H$-isotopic to the zero section $0_{T^*M}$, such that there exist two increasing sequences of integers $n_k$ and $m_k$ such that $(\mathcal{L}_{n_k})_{k\geq 0} = (\phi^{n_k}_H(\mathcal{L}))_{k\geq 0}$ and $(\mathcal{L}_{-m_k})_{k\geq 0} = (\phi^{-m_k}_H(\mathcal{L}))_{k\geq 0}$ respectively converge with reduced complexities to Lagrangian submanifolds $\mathcal{L}_\omega$ and $\mathcal{L}_\alpha$ which are $H$-isotopic to the zero section $0_{T^*M}$.\\

According to the Definition \ref{ReducedComplexity} of reduced complexity convergence, if we consider two Hamiltonian maps $\varphi_\alpha$ and $\varphi_\omega \in \Ham (T^*M, \omega)$ such that $\mathcal{L}_\alpha = \varphi_\alpha(0_{T^*M})$ and $\mathcal{L}_\omega = \varphi_\omega(0_{T^*M})$, then we have
\begin{enumerate} [label=\roman*.]
	\item $(\mathcal{L}_{-m_k})_{k\geq 0}$ and $(\mathcal{L}_{n_k})_{k\geq 0}$ converge respectively to $\mathcal{L}_\alpha$ and $\mathcal{L}_\omega$ in the Haussdorff topology.
	\item if we denote by $l^\alpha_k$ and $l^\omega_k$ respective Liouville primitives on $\varphi_\alpha^{-1}(\mathcal{L}_{-m_k})$ and $\varphi_\omega^{-1}(\mathcal{L}_{n_k})$, then we have $\lim_k \osc(l^\alpha_k) = \lim_k \osc(l^\omega_k) =0$.
\end{enumerate}

We choose the Liouville primitives $l^\alpha_k$ and $l^\omega_k$ on $\varphi_\alpha^{-1}(\mathcal{L}_{-m_k})$ and $\varphi_\omega^{-1}(\mathcal{L}_{n_k})$ as follows. The Hamiltonian maps $\varphi_\alpha$ and $\varphi_\omega$ are by definition exact symplectomorphisms of $(T^*M, \lambda)$ and there exist regular maps $f_\alpha$ and $f_\omega : T^*M \to \mathbb{R}$ such that
\begin{equation} 
	\varphi_\alpha^*\lambda - \lambda = df_\alpha \quad \text{and} \quad \varphi_\omega^*\lambda - \lambda = df_\omega
\end{equation}

Then, for all integer $k \in \mathbb{N}$, we choose the Liouville primitives following Lemma \ref{PrimitiveChange}.
\begin{equation} \label{Prim2}
	 l^\alpha_k=  h_{-m_k} \circ \varphi_\alpha - f_\alpha  \quad \text{and} \quad  l^\omega_k= h_{n_k}  \circ \varphi_\omega - f_\omega
\end{equation}\\

As in Section \ref{Extsection}, we extend $\mathcal{L}_\alpha$ and $\mathcal{L}_\omega$ respectively to exact Lagrangian submanifolds $\mathscr{L}_\alpha$ and $\mathscr{L}_\omega$ of $T^*\mathcal{M} = T^*(\mathbb{R} \times M)$. And we can successively construct :
\begin{enumerate}
	\item Liouville primitives $h_\alpha : \mathcal{L}_\alpha \to \mathbb{R}$ for $\mathcal{L}_\alpha$ and $h_\omega : \mathcal{L}_\omega \to \mathbb{R}$ for $\mathcal{L}_\omega$, with time evolution $h^\alpha_t$ and $h^\omega_t$ following (\ref{Primitive}).
	\item Liouville primitives $\mathpzc{h}_\alpha :\mathscr{L}_\alpha \to \mathbb{R}$ for $\mathscr{L}_\alpha$ and $\mathpzc{h}_\omega : \mathscr{L}_\omega \to \mathbb{R}$ for $\mathscr{L}_\omega$ with
	\begin{equation}
		\mathpzc{h}_\alpha(t,E_\alpha, q, p_\alpha) = h_{\alpha,t}(q,p_\alpha) \quad \text{and} \quad \mathpzc{h}_\omega(t,E_\omega, q, p_\omega) = h_{\omega,t}(q,p_\omega)
	\end{equation}
	whenever $(t,E_\alpha, q, p_\alpha) \in \mathscr{L}_\alpha$ and $(t,E_\omega, q, p_\omega) \in \mathscr{L}_\omega$.
	\item Graph selectors $u_\alpha : \mathbb{R} \times M \to \mathbb{R}$ for $\mathscr{L}_\alpha$ and $u_\omega : \mathbb{R} \times M \to \mathbb{R}$ for $\mathscr{L}_\omega$ respectively associated with $\mathpzc{h}_\alpha$ and $\mathpzc{h}_\omega$, following Subsection \ref{GSConstruction}. Proposition \ref{Graphext} provides two open full-measure subsets $\mathcal{U}_\alpha$ and $\mathcal{U}_\omega$ of $\mathbb{R} \times M$ where $u_\alpha$ and $u_\omega$ are respectively regular.
\end{enumerate}

Note that the Liouville primitives $l^\alpha = h_\alpha \circ \varphi_\alpha - f_\alpha$ and 
$l^\omega = h_\omega \circ \varphi_\omega - f_\omega$ of the zero section $\varphi_\alpha^{-1}( \mathcal{L}_\alpha ) = \varphi_\omega^{-1}( \mathcal{L}_\omega ) = 0_{T^*M}$ are constant. In this case, we have 
\begin{align*}
	\varphi^{-1}_\alpha(\mathcal{L}_{n_k}) \string # \overline{\varphi^{-1}_\alpha(\mathcal{L}_{\alpha})} =\varphi^{-1}_\alpha(\mathcal{L}_{n_k}) \quad \text{and} \quad  \varphi^{-1}_\omega(\mathcal{L}_{n_k}) \string # \overline{\varphi^{-1}_\omega(\mathcal{L}_{\omega})} =  \varphi^{-1}_\omega(\mathcal{L}_{n_k}) 
\end{align*}
with associated Liouville primitive $l^\alpha_k$ and $l^\omega_k$. This gives a meaning to $\osc(l^\alpha_k - l^\alpha) = \osc(l^\alpha_k)$ and $\osc(l^\omega_k - l^\omega) = \osc(l^\omega_k)$ where $l^\alpha$ and $l^\omega$ are considered to be elements of $\mathbb{R}$.

\subsection{Study of Limit Points} \label{LPsubsection}

Let us begin the proof of the main Theorem \ref{MainTheorem}. The most crucial step is dealt with in this subsection. The idea is to initially identify curves for which calibration can be easily established.\\

Let $x=(q,p)$ be a point of $\mathcal{L}$ and $x_\omega=(q_\omega, p_\omega)$ be a limit point of the sequence $\big(\phi_H^{n_k}(x)\big)_k$. We know from the Hausdorff convergence of $\mathcal{L}_{n_k}$ that the point $x_\omega$ belongs to $\mathcal{L}_\omega$. And set $x_\omega(t)=(q_\omega(t), p_\omega(t)) = \phi_H^t(x_\omega)$. We can assume, up to extraction, that 
\begin{equation} \label{Extract}
	\phi_H^{n_k}(x) \longrightarrow x_\omega \quad \text{and} \quad n_{k+1}-n_k \longrightarrow +\infty \quad \text{as } k \to \infty
\end{equation}

Similarly, let $x_\alpha=(q_\omega, p_\omega) \in \mathcal{L}_\alpha$ be a limit point of the sequence $\big(\phi_H^{-m_k}(x)\big)_k$. Set $x_\alpha(t)=(q_\alpha(t), p_\alpha(t)) = \phi_H^t(x_\alpha)$. And assume, up to extraction, that 
\begin{equation} \label{Extract2}
	\phi_H^{-m_k}(x) \longrightarrow x_\alpha \quad \text{and} \quad m_{k+1}-m_k \longrightarrow +\infty \quad \text{as } k \to \infty
\end{equation}

\begin{prop} \label{Calib}
	The curves $q_\alpha(t)$ and $q_\omega(t)$ are respectively calibrated by $u_\alpha$ and $u_\omega$.
\end{prop}

\begin{proof}
	We only prove the calibration for $q_\omega(t)$. The case of $q_\alpha(t)$ is done analogously. Let $a<b$ be two times and $b_k = a + n_{k+1}-n_k$. Set $u_k(t,q) = u(t + n_k, q)$, $x_k =(q_k,p_k)= \phi_H^{n_k}(x)$ and $x_k(t) = (q_k(t),p_k(t)) = \phi_H^t(x_k)$. In order to compute the calibration defect $\delta(u_\omega, q_{\omega|[a,b]})$, we need to link it to
	\begin{equation*}
		\delta(u_k, q_{k|[a,b]}) = \int_a^{b} L(s,q_k(s),\dot{q}_k(s)) \; ds - [u_k(b,q_k(b)) - u_k(a,q_k(a))]
	\end{equation*}
	This requires to know more on the asymptotic $C^1$ convergence of the curve $q_k$ and on the $C^0$ convergence of maps $u_k$ restricted to $[a,b] \times M$.
	
	\begin{lem} \label{LemDeltaConv}
		We have
		\begin{equation*}
			\delta(u_\omega, q_{\omega|[a,b]}) = \lim\limits_k  \delta(u_k, q_{k|[a,b]})
		\end{equation*}
	\end{lem}
	\begin{proof}
		Since the points $x_k=(q_k,p_k)$ converge to $x_\omega=(q_\omega,p_\omega)$, the curves $x_{k|[a,b]}$ converge uniformly to $x_{\omega|[a,b]}$. We also know from (\ref{Fenchel}) that $p_k(t) = \partial_v L(t,q_k(t),\dot{q}_k(t))$ and $p_\omega(t) = \partial_v L(t,q_\omega(t),\dot{q}_\omega(t))$. Thus the curves $q_{k|[a,b]}$ converge to $q_{\omega|[a,b]}$ in the $C^1$ topology so that 
		\begin{equation} \label{CalibDemu0}
			\lim\limits_k \int_a^{b} L(s,q_k(s),\dot{q}_k(s)) \; ds = \int_a^{b} L(s,q_\omega(s),\dot{q}_\omega(s)) \; ds
		\end{equation}	 
	
		For the comparison between the $u$ maps, we have for all time $t$, $u_k(t)$ and $u_\omega(t)$ are respective graph selectors of the Lagrangian submanifolds $\mathcal{L}_{t+n_k}$ and $\mathcal{L}_{\omega +t}:= \phi^t_H(\mathcal{L}_\omega)$. Then, applying Proposition \ref{Bound}, we have
		\begin{equation} \label{CalibDemu1}
			c( 1, \mathcal{L}_{t+n_k} \string # \overline{\mathcal{L}_{\omega +t}} ) \leq u_k(t) - u_\omega(t) \leq c( \mu, \mathcal{L}_{t+n_k} \string # \overline{\mathcal{L}_{\omega +t}} )
		\end{equation}
		and an application of Propositions \ref{Invariance} and \ref{inegh} gives
		\begin{equation}\label{CalibDemu2}
			c( 1, \mathcal{L}_{t+n_k} \string # \overline{\mathcal{L}_{\omega +t}}) = c( 1, \mathcal{L}_{n_k} \string # \overline{\mathcal{L}_{\omega}}) = c( 1, \varphi^{-1}_\omega(\mathcal{L}_{n_k}) \string # \overline{\varphi^{-1}_\omega(\mathcal{L}_{\omega})}) \geq \min (l^\omega_k -l^\omega)
		\end{equation}
		and 
		\begin{equation}\label{CalibDemu3}
			c( \mu, \mathcal{L}_{t+n_k} \string # \overline{\mathcal{L}_{\omega +t}} ) = c( \mu, \varphi^{-1}_\omega(\mathcal{L}_{n_k}) \string # \overline{\varphi^{-1}_\omega(\mathcal{L}_{\omega})}) \leq \max (l^\omega_k -l^\omega)
		\end{equation}
		where $l^\omega$ is the constant Liouville primitive on $\varphi^{-1}_\omega(\mathcal{L}_{\omega}) = 0_{T^*M}$. Gathering these three inequalities (\ref{CalibDemu1}), (\ref{CalibDemu2}) and (\ref{CalibDemu3}), we obtain for all $(t,q) \in \mathbb{R} \times M$
		\begin{equation} \label{ConstantDem0}
			\min (l^\omega_k -l^\omega) \leq  u_k(t,q) - u_\omega(t,q) \leq \max (l^\omega_k -l^\omega)
		\end{equation}
		Applying these inequalities at $t=a$ and $t=b$ yields
		\begin{equation}
			\big| [ u_k(b,q_k(b)) - u_\omega(b,q_k(b)) ] - [ u_k(a,q_k(a)) - u_\omega(a,q_k(a)) ]    \big| \leq \osc(l^\omega_k) \longrightarrow 0 \quad \text{as } k \to \infty
		\end{equation}
		Hence, we get
		\begin{multline} \label{CalibDemu4}
			\big| [ u_k(b,q_k(b)) - u_k(a,q_k(a)) ] - [ u_\omega(b,q_\omega(b)) - u_\omega(a,q_\omega(a)) ]    \big| \\
			\begin{split}
				&\leq \big| [ u_k(b,q_k(b)) - u_k(a,q_k(a)) ] - [ u_\omega(b,q_k(b)) - u_\omega(a,q_k(a)) ] \big| \\
				&+ \big| u_\omega(b,q_\omega(b)) - u_\omega(b,q_k(b)) \big| + \big|  u_\omega(a,q_\omega(a)) - u_\omega(a,q_k(a)) \big| \\
				&\leq \osc (l^\omega_k)  + \big| u_\omega(b,q_\omega(b)) - u_\omega(b,q_k(b)) \big| + \big|  u_\omega(a,q_\omega(a)) - u_\omega(a,q_k(a)) \big| 
			\end{split}
		\end{multline}
		Moreover, we know from the convergence of the curves $q_{k|[a,b]}$ to $q_{\omega|[a,b]}$ and the continuity of $u_\omega$ that
		\begin{align}\label{CalibDemu5}
			\lim_k u_\omega(b,q_k(b)) = u_\omega(b,q_\omega(b)) \quad \text{and} \quad \lim_k u_\omega(a,q_k(a)) = u_\omega(a,q_\omega(a))
		\end{align}
		Hence, we deduce from (\ref{CalibDemu4}) and (\ref{CalibDemu5}) that
		\begin{equation}\label{CalibDemu6}
			\lim_k u_k(b,q_k(b)) - u_k(a,q_k(a)) =  u_\omega(b,q_\omega(b)) - u_\omega(a,q_\omega(a))
		\end{equation}
		Gathering (\ref{CalibDemu0}) and (\ref{CalibDemu6}), we conclude that
		\begin{align*}
			\lim_k \delta(u_k, q_{k|[a,b]})  &= \lim_k \int_a^{b} L(s,q_k(s),\dot{q}_k(s)) \; ds - \lim_k [u_k(b,q_k(b)) - u_k(a,q_k(a))] \\
			&= \int_a^{b} L(s,q_\omega(s),\dot{q}_\omega(s)) \; ds + [u_\omega(b,q_\omega(b)) - u_\omega(a,q_\omega(a))] = \delta(u_\omega, q_{\omega|[a,b]})
		\end{align*}
	\end{proof}
	From assumption (\ref{Extract}), we have inclusion $[a,b] \subset [a,b_k]$ for large $k$. Thus, the lemma and the positivity of the defect of calibration (Proposition \ref{Domination}) lead to
	\begin{equation}
		0 \leq \delta(u_\omega, q_{\omega|[a,b]}) = \lim\limits_k \delta(u_k, q_{k|[a,b]}) \leq \liminf_k \delta(u_k, q_{k|[a,b_k]})
	\end{equation}
	Let us evaluate
	\begin{align*}
		\delta(u_k,q_{k|[a,b_k]}) & = \int_a^{b_k} L(s,q_k(s),\dot{q}_k(s)) \; ds - [u_k(b_k,q_k(b_k)) - u_k(a,q_k(a))]
	\end{align*}
	From the Fenchel's equality case, we know that for all time $s \in \mathbb{R}$,
	\begin{equation}
		L(s+n_k,q_k(s),\dot{q}_k(s)) = p_k(s).\dot{q}_k(s) - H(s+n_k,q_k(s), p_k(s))
	\end{equation}
	Set $\zeta_k(t) = (t+ n_k, E_k(t),q_k(t),p_k(t))$ to be a curve in $\mathscr{L}$. Since $\mathscr{L} \subset \{ \mathscr{H}=0 \}$ ,We have
	\begin{equation}
		E_k(t) = - H\big(t + n_k,q_k(t), p_k(t)\big)
	\end{equation}
	Therefore, using the time one periodicity of $L$, we get 
	\begin{align*}
		\int_a^{b_k} L(s,q_k(s),\dot{q}_k(s)) \; ds &= \int_a^{b_k} L(s+n_k,q_k(s),\dot{q}_k(s)) \; ds \\
		&= \int_a^{b_k} p_k(s).\dot{q}_k(s) - H(s+n_k,q_k(s), p_k(s)) \; ds \\
		&=  \int_a^{b_k} p_k(s).\dot{q}_k(s) + E_k(s) \; ds = \int_\zeta \Lambda_{|\mathscr{L}} =\int_\zeta d \mathpzc{h} \\
		&= \mathpzc{h}(\zeta(b_k)) - \mathpzc{h}(\zeta(a)) = h_{a+n_{k+1}}(x_k(b_k)) - h_{a+n_k}(x_k(a))
	\end{align*}	
	To simplify the notation, set $x^a_k =(q^a_k, p^a_k)= x_k(a)$ and note that 
	\begin{align*}
		x_k(b_k)= x_k(a+n_{k+1}-n_k) = \phi_H^{a+n_{k+1}-n_k}(x_k) = \phi_H^{a}(x_{k+1})=x^a_{k+1}
	\end{align*}		
	and
	\begin{align*}
		u_k(b_k) = u(b_k+n_k) = u(a+n_{k+1}) = u_{k+1}(a)
	\end{align*}
	Hence
	\begin{equation} \label{defect}
		\delta(u_k,q_{k|[a,b_k]}) = [h_{a+n_{k+1}}(x_{k+1}^a) - h_{a+n_k}(x_k^a)] - [u_{k+1}(a,q^a_{k+1}) - u_k(a,q^a_k)]
	\end{equation}\\
	
	We insert the terms $h^\omega_a(x^a_\omega)$ and $u_\omega(a,q^a_\omega)$ as follows
	\begin{equation} \label{CalibDemDelta0}
		\begin{split}
			\delta(u_k,q_{k|[a,b_k]}) &= [h_{a+n_{k+1}}(x_{k+1}^a) - h^\omega_a(x^a_\omega)] - [u_{k+1}(a,q^a_{k+1})  - u_\omega(a,q^a_\omega)] \\ 
			& + [h^\omega_a(x^a_\omega) - h_{a+n_k}(x_k^a)] - [u_{\omega}(a,q^a_\omega) - u_{k}(a,q^a_k) ] 
		\end{split}
	\end{equation}
	We need to prove that the sequence 
	\begin{equation} \label{CalibDemDelta1}
		\delta_k = [h_{a+n_k}(x_k^a) - h^\omega_a(x^a_\omega)] - [u_{k}(a,q^a_k) - u_{\omega}(a,q^a_\omega) ]
	\end{equation}
	converges to $0$. This involves comparing $u_{k+1} - u_{k}$ and $"h_{a+n_{k+1}} -  h_{a+n_k}"$ where the second term does not make sense, justifying the need to use the primitives $l_t$ introduced in (\ref{Prim2}). 
	
	Let $c_k$ be any value in the image of  $u_{k}(a) - u_\omega(a)$. We infer from the inequalities (\ref{ConstantDem0}) that
	\begin{equation}
		\Vert l^\omega_{k} - l^\omega - c_k \Vert _\infty \leq Osc(l^\omega_{k}) \quad \text{and} \quad \Vert u_k(a) - u_\omega(a) - c_k \Vert _\infty \leq Osc(u_k(a) - u_\omega(a)) \leq Osc(l^\omega_k)
	\end{equation}			
	We get 
	\begin{align} \label{CalibDemDelta2}
		|\delta_k| & \leq | h_{a+n_k}(x_k^a) - h^\omega_a(x^a_\omega) - c_k| + |u_{k}(a,q^a_k) - u_{\omega}(a,q^a_\omega) - c_k| 
	\end{align}
	The choice of the constant $c_k$ leads to
	\begin{align*}
		|u_{k}(a,q^a_k) - u_{\omega}(a,q^a_\omega) - c_k| &\leq |u_k(a,q^a_\omega) - u_{\omega}(a,q^a_\omega) - c_k| + | u_{\omega}(a,q^a_\omega) -  u_\omega(a,q^a_k) | \\
		& \leq Osc(u_k(a) - u_\omega(a)) + | u_{\omega}(a,q^a_\omega) -  u_\omega(a,q^a_k) | \\
		& \leq Osc(l_k)+ | u_{\omega}(a,q^a_\omega) -  u_\omega(a,q^a_k) | \longrightarrow 0 \quad \text{ as } k \to \infty
	\end{align*}
	where we used the continuity of $u_\omega$. For the remaining term of (\ref{CalibDemDelta2}), we first use the identity (\ref{Prim}) on $h$ and $h^\omega$ to get
	\begin{multline*}
		| h_{a+n_k}(x_k^a) - h^\omega_a(x^a_\omega) - c_k| \leq \left|\int_0^a (x_k^*\lambda - x_{\omega}^*\lambda) - \int_0^a H(\tau, x_k(\tau)) - H(\tau, x_{\omega}(\tau)) \; d\tau \right| \\ + | h_{n_k}(x_k) - h^\omega_0(x_\omega) - c_k|
	\end{multline*}
	Since the curve $x_k$ converges to $x_\omega$ uniformly on $[0,a]$, and since by the identity \eqref{LegendreMapsEL} the curve $\dot{q}_k$ converges uniformly to $q_\omega$ on the same interval $[0,a]$, we infer that
	\begin{align*}
		\left|\int_0^a (x_k^*\lambda - x_{\omega}^*\lambda) - \int_0^a H(\tau, x_k(\tau)) - H(\tau, x_{\omega}(\tau)) \; d\tau \right|  \longrightarrow 0 \quad \text{ as } k \to \infty
	\end{align*}
	Set $\tilde{x}_{k} = \varphi^{-1}_\omega(x_{k})$ and $\tilde{x}_{\omega} = \varphi^{-1}_\omega(x_{\omega})$. Then, applying (\ref{Prim2}) yields
	\begin{align*}
		| h_{n_k}(x_k) - h^\omega_0(x_\omega) - c_k| & \leq | l^\omega_k(\tilde{x}_k) - l^\omega - c_k| + |f_\omega(\tilde{x}_k) - f_\omega(\tilde{x}_\omega)| \\
		&\leq Osc( l^\omega ) +|f_\omega(\tilde{x}_k) - f_\omega(\tilde{x}_\omega)| \longrightarrow 0 \quad \text{ as } k \to \infty
	\end{align*}
	where we used the continuity of the map $f_\omega$. This establishes the limit $\lim_k \delta_k =0$ and hence
	\begin{align*}
		0 \leq \delta(u_\omega,q_{|[a,b]}) \leq \liminf_k \delta(u_k,q_{k|[a,b_k]}) = \liminf_k (\delta_{k+1} - \delta_k) =0
	\end{align*}
\end{proof}

\subsection{Study of all Points} \label{APsubsection}

We now extend the previous $u_\omega$-calibration result to the $u$-calibration of every Hamiltonian curve included in $\mathscr{L}$.

\begin{prop} \label{Calibration}
	Let $x = (q,p)$ be a point of $\mathcal{L}$ and $x(t) = (q(t),p(t)) = \phi_H^t(x)$. Then the curve $q(t)$ is calibrated by $u$.
\end{prop}

In the proof of the Main Theorem \ref{MainTheorem}, we will see that this calibration, by Fathi's result \ref{Fathi} on calibrated curves, ensures that the submanifolds $\mathcal{L}_t$ are graphs over the zero section of $T^*M$ for all times $t \in \mathbb{R}$.            

\begin{proof}
	We use the notation $x_t = (q_t,p_t)$ instead of $x(t) = (q(t),p(t))$. Let $a<b$ be two real times. We will estimate the defect of calibration $\delta(u,q_{|[a,b]})$. But first, let $x_\alpha=(q_\alpha,p_\alpha)\in \mathcal{L}_\alpha$ and $x_\omega=(q_\omega,p_\omega) \in \mathcal{L}_\omega$ be respective limit points of the sequences $\big(\phi_H^{-m_k}(x)\big)_{k\geq 0}$ and $\big(\phi_H^{n_k}(x)\big)_{k\geq 0}$. For $k$ large enough, $[a,b] \subset [-m_k,n_k]$, hence
	\begin{equation}
		0 \leq \delta(u,q_{|[a,b]}) \leq \liminf_k \delta(u,q_{|[-m_k,n_k]})
	\end{equation}
	and from the same computation that led to (\ref{defect}), we get
	\begin{equation}
		\begin{split}
			\delta(u,q_{|[-m_k,n_k]}) & = [h_{n_k}(x_{n_k}) - h_{-m_k}(x_{-m_k})] - [u(n_k,q_{n_k}) - u(-m_k,q_{-m_k}) ]\\
			&= [h_{n_k}(x_{n_k}) - u(n_k,q_{n_k})] - [h_{-m_k}(x_{-m_k}) -  u(-m_k,q_{-m_k}) ]
		\end{split}
	\end{equation}
	Let $\delta^k_\omega$ and $\delta^k_\alpha$ be defined as 
	\begin{align*}
		\delta^k_\alpha &=[h_{-m_k}(x_{-m_k}) -  u(-m_k,q_{-m_k})] - [h^\alpha_0(x_\alpha) - u_\alpha(0,q_\alpha)] \\
		&= [h_{-m_k}(x_{-m_k}) - h^\alpha_0(x_\alpha)] - [  u(-m_k,q_{-m_k})- u_\alpha(0,q_\alpha)]
	\end{align*}	
	and
	\begin{align*}
		\delta^k_\omega &=[h_{n_k}(x_{n_k}) - u(n_k,q_{n_k})] - [h^\omega_0(x_\omega) - u_\omega(0,q_\omega)] \\
		&= [h_{n_k}(x_{n_k}) - h^\omega_0(x_\omega)] - [u(n_k,q_{n_k})- u_\omega(0,q_\omega)]
	\end{align*}
	we have $x_{-m_k}$ converges to $x_\alpha$ and $x_{n_k}$ converges to $x_\omega$. Thus the same control established for (\ref{CalibDemDelta1}) results in the limits
	\begin{equation}
		\delta^k_\omega \longrightarrow 0 \quad \text{and} \quad \delta^k_\alpha \longrightarrow 0 \quad \text{as } k \to \infty
	\end{equation}
	As a result, $\delta(u,q_{|[-m_k,n_k]})$ converges to $\delta_\infty$ given by
	\begin{equation}
		\delta_\infty  = [h^\omega_0(x_\omega) - u_\omega(0,q_\omega)] - [h^\alpha_0(x_\alpha) - u_\alpha(0,q_\alpha)]
	\end{equation}

	We claim the following lemma
	\begin{lem} \label{ConvexityLemma}
		We have
		\begin{equation}
			u_\alpha(0,q_\alpha) = h^\alpha_0(x_\alpha) \quad \text{and} \quad u_\omega(0,q_\omega) = h^\omega_0(x_\omega)
		\end{equation}
	\end{lem}
	This lemma leads to the nullity of $\delta_\infty$ and consequently to
	\begin{equation}
		0 \leq \delta(u,q_{|[a,b]}) \leq \liminf_k \delta(u,q_{|[-m_k,n_k]}) = \delta_\infty = 0 
	\end{equation}
	which implies the calibration.
\end{proof}
	
\begin{proof}[Proof of Lemma \ref{ConvexityLemma}]
	We only prove the equality for $x_\alpha= (q_\alpha, p_\alpha)$. Recall from Proposition \ref{Calib} that $q_\alpha(t)$ is $u_\alpha$-calibrated. Then we can apply Proposition \ref{Fathi} at time zero to get that $u_\omega(t,q)$ is differentiable at $(0,q_\alpha)$ and 
	\begin{equation}
		d_q u_\alpha(0,q_\alpha) = \partial_v L(0,q_\alpha, \dot{q_\alpha}(0))=p_\alpha \quad \text{and} \quad  \mathscr{H}(0,\partial_t u_\alpha(0,q_\alpha), q_\alpha, p_\alpha)=0
	\end{equation}
	This implies that 
	\begin{equation} \label{ConvexityLemmaDem1}
		du_\alpha(0,q_\alpha)= (\partial_t u_\alpha(0,q_\alpha), d_qu_\alpha(0,q_\alpha)) \in \{\mathscr{H}=0 \} \cap T^*_{(0,q_\alpha)}(\mathbb{R} \times M)
	\end{equation}
	
	Moreover, We know from Lemma \ref{ConvexHull} that 
	\begin{align*}
		du_\alpha(0,q_\alpha) = (\partial_t u_\alpha(0,q_\alpha), d_qu_\alpha(0,q_\alpha)) \in C_{u_\alpha}(0,q_\alpha) \subset \{\mathscr{H}\leq 0 \} \cap T^*_{(0,q_\alpha)}(\mathbb{R} \times M)
	\end{align*}
	where the last inclusion, due to the strict convexity of the set $\{\mathscr{H}=0\} \cap T^*_{(0,q_\alpha)}(\mathbb{R} \times M)$, has been established in the proof of Proposition \ref{Fathi}. This strict convexity tells that $\{\mathscr{H}=0\} \cap T^*_{(0,q_\alpha)}(\mathbb{R} \times M)$ is the set of extremal points of its convex hull $ \{\mathscr{H}\leq 0 \} \cap T^*_{(0,q_\alpha)}(\mathbb{R} \times M)$. Thus, we infer from the inclusion (\ref{ConvexityLemmaDem1}) that $du_\alpha(0,q_\alpha)$ is an extremal point of the convex set $C_{u_\alpha}(0,q_\alpha)$, or in other words that $du_\alpha(0,q_\alpha) \in K_{u_\alpha}(0,q_\alpha)$. By definition of the set $K_{u_\alpha}(0,q_\alpha)$, there exists a sequence $(t_n,q_n)$ in the set $\mathcal{U}_\alpha$ introduced in Proposition \ref{Graphext} converging to $(0,q_\alpha)$ such that $du_\alpha(t_n,q_n)$ converges to $du_\alpha(0,q_\alpha) = (\partial_tu_\alpha(0,q_\alpha), p_\alpha )$.
	
	By definition of the set $\mathcal{U}_\alpha$, we have 
	\begin{align*}
		\big(t_n, \partial_tu_\alpha(t_n,q_n),q_n,d_qu_\alpha(t_n,q_n)\big) \in \mathscr{L}_\alpha \quad  \text{and} \quad  u_\alpha(t_n,q_n) = \mathpzc{h}_\alpha\big(t_n, \partial_tu_\alpha(t_n,q_n),q_n,d_qu_\alpha(t_n,q_n)\big)
	\end{align*}
	We conclude that $\big(0, \partial_tu_\alpha(0,q_\alpha),q_\alpha,p_\alpha \big) \in \mathscr{L}_\alpha$ and, by continuity of the maps $u_\alpha$ and $\mathpzc{h}_\alpha$, that 
	\begin{align*}
		u_\alpha(0,q_\alpha)&= \lim\limits_n u_\alpha(t_n,q_n) = \lim\limits_n \mathpzc{h}_\alpha\big(t_n, \partial_tu_\alpha(t_n,q_n),q_n,d_qu_\alpha(t_n,q_n)\big) \\
		&= \mathpzc{h}_\alpha\big(0, \partial_tu_\alpha(0,q_\alpha),q_\alpha,p_\alpha \big) = h^\alpha_0(q_\alpha,p_\alpha) = h^\alpha_0(x_\alpha)
	\end{align*}
\end{proof}

\subsection{Proofs of the Main Results}

We have seen in Propositions \ref{Calib} and \ref{Calibration} that reduced complexity convergence in both positive and negative times implies that the extended Lagrangian submanifold $\mathscr{L}$ is foliated by calibrated curves. By applying Fathi's Theorem \ref{Fathi} on calibrated curves, it will follow that $\mathscr{L}$ is a graph over the zero section of $T^*(\mathbb{R} \times M)$.      

\begin{proof}[Proof of Theorem \ref{MainTheorem}]
	Let $t \in \mathbb{R}$ be a fixed time. For all $x = (q,p)$ in $\mathcal{L}_t$ and $x(s) = (q(s),p(s)) = \phi_H^{t,s}(q,p)$, we know from Proposition \ref{Calibration} that $q(s)$ is calibrated by $u$. Hence, we infer from Proposition \ref{Fathi} that $u$ is differentiable at $(t,q)$ and $d_qu(t,q) = \partial_v L(t,q,\dot{q}(t)) = p$. Thus, if we denote by $\mathcal{G}(du_t)$ the graph of $du_t=d_qu(t,\cdot)$ in $T^*M$, we get the inclusion $\mathcal{L}_t \subset \mathcal{G}(du_t)$. And since the projection $\mathcal{L}_t \to M$ is onto, we conclude that $\mathcal{L}_t = \mathcal{G}(du_t)$ is a graph over $0_{T^*M}$.
	
	In order to obtain regularity, using the notation of Proposition \ref{Fathi}, we proved that for all $\varepsilon >0$, we have  $A_{\varepsilon,u} = \mathbb{R} \times M$. Hence, $du$ is locally Lipschitz on $\mathbb{R} \times M$. And since the Lagrangian submanifolds $\mathcal{L}_t$ are $C^1$ regular, we conclude that they are $C^1$ graphs over the zero section $0_{T^*M}$. 
\end{proof}

\begin{proof}[Proof of Corollary \ref{MainCor2}]
	The proof of this result is based on tools coming from the weak-KAM theory. We will show that the graph selector $u(t,x)$ is a recurrent (viscosity) solution of the Hamiltonian-Jacobi equation (\ref{HJ}) associated with the Hamiltonian $H$. We present the concepts in the Appendix \ref{AppendixLO}.
	
	All the proofs of the paper remain unchanged by adding a constant to the Hamiltonian $H$. Hence, up to considering $H-\alpha_0$, we can assume that the \Mane critical value $\alpha_0$ is null. We consider the positive Lax-Oleinik operator $\mathcal{T}^{s,t}_+$ and we show that for all times $s>t$, $\mathcal{T}^{s,t}_+u(s) = u(t)$. Fix a point $q$ in $M$ and two times $s>t$. We saw in Proposition \ref{Calibration} that there exists a $u$-calibrated curve $q(\tau)$ with $q(t) =q$. Then, we have
	\begin{equation*}
		u(t,q) = u(s,q(s)) - h^{t,s}(q,q(s)) \leq \mathcal{T}^{s,t}_+u(s)(q) = \sup_{q' \in M} \big\{ u(s,q') - h^{t,s}(q,q') \big\}
	\end{equation*}
	Moreover, we know from Proposition \ref{Domination} that for any point $q'$ in $M$, we have 
	\begin{align*}
		u(s,q') - u(t,q) \leq h^{s,t}(q,q')
	\end{align*}
	yielding 
	\begin{align*}
		 \mathcal{T}^{s,t}_+u(s)(q) = \sup_{q' \in M} \big\{ u(s,q') - h^{t,s}(q,q') \big\} \leq u(t,q)
	\end{align*}
	We obtained the equality $u(t) = \mathcal{T}^{s,t}_+u(s)$.
	
	Fix a point $q_0$ in $M$. By Proposition \ref{BoundedLO} the family $\big(\mathcal{T}_+^{s,t}u(s)\big)_{t\leq s-1}$ is uniformly bounded for the $C^0$ norm. Hence, we can assume up to extraction that the sequence $u(-m_n,q_0)$ converges to a limit $c_\alpha \in \mathbb{R}$. Adding a constant to $u_\alpha$, we can also assume that $u_\alpha(q_0) = c_\alpha$. Then, since we have already established that $\lim_n \osc (u_{-m_n} - u_\alpha) =0$, we obtain the convergence of $u_{-m_n}$ to $u_\alpha$ in the $C^0$-topology.
	
	We further assume up to extraction that the sequence $m_{n+1} - m_n$ diverges to infinity. We consider the negative Lax-Oleinik operator $\mathcal{T}_-^{s,t}$ with $\mathcal{T}_-^t := \mathcal{T}_-^{0,t}$. Analogously to the above, the map $u$ verifies for all times $s<t$, $u(t) = \mathcal{T}_-^{s,t}u(s)$. Therefore, the non-expansiveness of the $\mathcal{T}_-^{s,t}$ given by Proposition \ref{NonExpansiveness} leads to
	\begin{align*}
		\Vert u(-m_{n+1} + m_n) - u(0) \Vert _\infty  & = \Vert  \mathcal{T}_-^{-m_{n+1} + m_n} u(0) - u(0) \Vert _\infty \\
		& = \Vert \mathcal{T}_-^{m_n} u_{-m_{n+1}} - \mathcal{T}_-^{m_n}u_{-m_n} \Vert _\infty   \\
		& \leq \Vert u_{-m_{n+1}} - u_{-m_n} \Vert _\infty \\
		& \leq \Vert u_{-m_{n+1}} - u_\alpha \Vert _\infty + \Vert u_\alpha - u_{-m_n} \Vert _\infty  \longrightarrow 0 \quad \text{as } n \to \infty
	\end{align*}
	This shows that $u$ is a recurrent map. In particular, it is possible to take $u_\alpha$ and $u_\omega$ to be equal to $u_0$. The definition of reduced complexity convergence provides the respective Hausdorff convergence of the graphs $\mathcal{L}_{-m_k}$ and $\mathcal{L}_{n_k}$ of $du_{-m_k}$ and $du_{n_k}$ to the graph $\mathcal{L}$ of $du_0=du_\alpha=du_\omega$. This concludes the proof of the $C^0$-recurrence of the Lagrangian $\mathcal{L}$.
\end{proof} 

\begin{rem}
	More generally, this Corollary is a direct consequence of Theorem \ref{MainTheorem} and the following Proposition 
	\begin{prop}
		Let $u:\mathbb{R} \times M \to \mathbb{R}$ be a $C^1$ global solution of the Hamilton-Jacobi equation 
		\begin{equation}
			\partial_tu + H(t,q,d_qu(t,q)) = \alpha_0
		\end{equation}
		where $\alpha_0$ is the critical \Mane value introduced in Proposition \ref{ManeCritValueProp}. Then $u$ is $C^1$-recurrent in positive and negative (integer) times.
	\end{prop}
	
	The proof of the $C^0$-recurrence is included in the proof of Corollary \ref{MainCor2}. However, the $C^1$-recurrence is more intricate and follows from a result by M-C.Arnaud \cite{MR2150356} which states that if $\mathcal{T}_{\pm}^{n_k}u$ converges in $C^0$-topology to a scalar map $v$, then the graph of $d\mathcal{T}_{\pm}^{n_k}u$ converges for the Hausdorff distance to the graph of $dv$.
\end{rem}

\appendix

\section{The Positive Lax-Oleinik Operator} \label{AppendixLO}

This appendix is devoted to defining the positive Lax-Oleinik operator and discussing two of its basic properties. An exposition on the analogous negative Lax-Oleinik operator in the non-autonomous framework can be found in \cite{Representation}. The negative operator generates what we call viscosity solutions of the Hamilton-Jacobi equation, while the positive operator introduces new objects. Examples can still be constructed using Peierls barriers, which we define below.

\begin{defi}
	Fix two times $s<t$.
	\begin{enumerate}
		\item The \textit{potential} $h_0^{s,t}: M \times M \to \mathbb{R}$ is defined as
		\begin{equation} \label{Potential0}
			h_0^{s,t} (x,y) = \inf \left\{ \int_s^t L(\tau,\gamma(\tau), \dot{\gamma}(\tau)) \; d\tau \; \left| \;
			\begin{matrix}
				\gamma : & [s,t] \to M \\
				& s \mapsto x \\
				& t \mapsto y
			\end{matrix} \right.	\right\}
		\end{equation}
		where the infimum is taken over such absolutely continuous curves $\gamma$.
		
		\item For any constant $c \in \mathbb{R}$, we consider the \textit{$(L+c)$-potential} $h_c^{s,t}: M \times M \to \mathbb{R}$ defined by
		\begin{equation} \label{Potentialc}
			h_c^{s,t} = h_0^{s,t} + c.(t-s)
		\end{equation}
		
		\item The \textit{Peierls Barrier} $h_c^{s,\infty+t} : M \times M \to \mathbb{R}$ is defined as
		\begin{equation}\label{PeierlsBarrier}
			h_c^{s,\infty+t}(x,y) = \liminf_n h_c^{s,n+t}(x,y)
		\end{equation}
	\end{enumerate}
\end{defi}

The following proposition justifies the definition of the Peierls barrier as a non trivial object.
\begin{prop} \label{ManeCritValueProp}
	There exists a real constant $\alpha_0$, known as the \textit{\Mane critical value}, such that
	\begin{enumerate}[label=\roman*.]
		\item If $c < \alpha_0$, then for all times $s,t \in \mathbb{R}$, $h_c \equiv -\infty $.
		\item If $c > \alpha_0$, then for all times $s,t \in \mathbb{R}$, $h_c \equiv +\infty $.
		\item The Peierls Barrier $h^{s,\infty+t}:= h^{s,\infty+t}_{\alpha_0}$ is everywhere finite and Lipschitz on $M \times M$.
	\end{enumerate}
\end{prop}

We now define the Lax-Oleinik operators. The set $\mathcal{C}^0(M,\mathbb{R})$ denotes the set of continuous scalar maps $u_0 : M \to \mathbb{R}$.

\begin{defi}
	Fix two times $s<t$,
	\begin{enumerate}
		\item The \textit{negative Lax-Oleinik operator} $\mathcal{T}_{-,0}^{s,t} : \mathcal{C}^0(M,\mathbb{R}) \to \mathcal{C}^0(M,\mathbb{R})$ is defined by
		\begin{equation} \label{LO-}
			\begin{split}
				\mathcal{T}_{-,0}^{s,t} u_0(x) & 
				= \inf_{
				\begin{matrix}
					\gamma : [s,t] \rightarrow M \\
					t \mapsto x 
				\end{matrix}
				} \left\{  u_0(\gamma(s)) + \int_s^t L(\tau,\gamma(\tau), \dot{\gamma}(\tau)) \;d\tau \right\} \\
				& = \inf_{y \in M} \left\{  u_0(y) + h_0^{s,t}(y,x) \right\}\\
			\end{split}
		\end{equation}
		
		\item The \textit{positive Lax-Oleinik operator} $\mathcal{T}_{+,0}^{t,s} : \mathcal{C}^0(M,\mathbb{R}) \to \mathcal{C}^0(M,\mathbb{R})$ is defined by
		\begin{equation} \label{LO+}
			\begin{split}
				\mathcal{T}_{+,0}^{t,s} u_0(x) & 
				= \sup_{
				\begin{matrix}
					\gamma : [s,t] \rightarrow M \\
					s \mapsto x 
				\end{matrix}
				} \left\{  u_0(\gamma(t)) - \int_s^t L(\tau,\gamma(\tau), \dot{\gamma}(\tau)) \;d\tau \right\} \\
				& = \sup_{y \in M} \left\{  u_0(y) - h_0^{s,t}(x,y) \right\}\\
			\end{split}
		\end{equation}
		
		\item The \textit{full Lax-Oleinik operators} $\mathcal{T}_-^{s,t}, \mathcal{T}_+^{t,s} : \mathcal{C}^0(M,\mathbb{R}) \to \mathcal{C}^0(M,\mathbb{R})$ are defined by
		\begin{equation}\label{LOfull}
			\mathcal{T}_-^{s,t} u_0 = \mathcal{T}_{-,0}^{s,t} u_0 + \alpha_0.(t-s) \quad \text{and} \quad \mathcal{T}_+^{t,s} u_0 = \mathcal{T}_{+,0}^{t,s} u_0 - \alpha_0.(t-s)
		\end{equation}
	\end{enumerate}
	We omit the notation of the first time whenever it is null, i.e $\mathcal{T}_\pm ^{t} := \mathcal{T}_\pm ^{0,t}$ and $\mathcal{T}_\pm ^{t} := \mathcal{T}_\pm ^{0,t}$.
\end{defi}

As mentioned above, the negative Lax-Oleinik operator generates what is called the viscosity solutions $u(t,x) = \mathcal{T}_-^{0,t}u_0(x)$ of the Hamilton-Jacobi equation
\begin{equation} \label{HJ}
	\partial_tu + H(t,x,d_xu) = \alpha_0
\end{equation}
The positive Lax-Oleinik operator introduces a new type of weak-solutions that were unknown before the development of A. Fathi's weak-KAM theory.\\

We present a first dynamical property without proof.

\begin{prop} \label{NonExpansiveness}
	Fix two times $s<t$. The Lax-Oleinik operators $\mathcal{T}_-^{s,t}$ and $\mathcal{T}_+^{t,s}$ are non-expansive, i.e for all continuous scalar maps $u$ and $v$ in $\mathcal{C}^0(M, \mathbb{R})$, we have
	\begin{equation}
		\Vert \mathcal{T}_-^{s,t} u - \mathcal{T}_-^{s,t}v \Vert _\infty  \leq \Vert  u - v \Vert _\infty  \quad \text{and} \quad  \Vert \mathcal{T}_+^{t,s} u - \mathcal{T}_-^{t,s}v \Vert _\infty  \leq \Vert  u - v \Vert _\infty 
	\end{equation}
\end{prop}

Although analogous to the case of the negative Lax-Oleinik operator, the proofs of the following propositions are, to our knowledge, lacking in the literature. Therefore, we will present them in full detail.

\begin{prop} \label{WeakKAM+}
	Fix a point $y \in M$. The map $u(t,x) = - h^{t,\infty}(x,y) + \alpha_0.t$ is a \textit{positive weak-KAM solution} of the Hamilton-Jacobi equation. More precisely, for all times $s<t$, $\mathcal{T}_+^{t,s}u(t, \cdot) = u(s, \cdot)$ and $\mathcal{T}_+^{0,-1} u(0, \cdot) = u(0,\cdot)$.
\end{prop}

\begin{proof}
	Let $v(t,x) = -h^{t,\infty}(x,y)$. We claim that for all times $s<t$, $\mathcal{T}_{+,0}^{t,s} v(t, \cdot) = v(s, \cdot)$. Indeed, we prove it establishing a double inequality. 
	
	Fix two times $s<t$ and a point $x \in M$. Let $k_n$ be an increasing sequence of integers such that $h^{s,\infty}(x,y) = \lim_n h^{s,k_n}(x,y)$. For every integer $n$, the compactness of $M$ ensures the existence of a point $z_n \in M$ realizing the following infimum
	\begin{equation*}
		h^{s,t}(x,z_n) + h^{t,k_n}(z_n,y) = \inf_{z \in M} \{ h^{s,t}(x,z) + h^{t,k_n}(z,y) \}
	\end{equation*}		
	Moreover, it is easy to see that
	\begin{equation*}
		\inf_{z \in M} \{ h^{s,t}(x,z) + h^{t,k_n}(z,y) \} = h^{s,k_n}(x,y)
	\end{equation*}
	Hence, after assuming, up to extraction, that the sequence $z_n$ converges to a point $z \in M$, we deduce that
	\begin{align*}
		h^{s,\infty}(x,y) & = \lim\limits_n h^{s,k_n}(x,y) = \lim\limits_n h^{s,t}(x,z_n) + \liminf_n   h^{t,k_n}(z_n,y)  \\
		&= h^{s,t}(x,z) + \liminf_n  h^{t,k_n}(z_n,y) \\
		& \geq  h^{s,t}(x,z) + h^{t,\infty}(z,y)
	\end{align*}	
	where we mainly used properties of Peierls barrier that can be found in Proposition 5.2 and 5.3 of \cite{Representation}. This yields a first inequality
	\begin{align*}
		v(s,x) = - h^{s,\infty}(x,y) & \leq  - h^{t,\infty}(z,y) - h^{s,t}(x,z)  \\
		& \leq \mathcal{T}_{+,0}^{t,s} v(t)(x) = \sup_{z'\in M} \big\{ - h^{t,\infty}(z',y) - h^{s,t}(x,z') \big\}
	\end{align*}
	
	We now establish the inverse inequality. We know that for any point $z$ of $M$ and any large integer $n >s$, 
	\begin{align*}
		h^{s,n}(x,y) \leq h^{s,t}(x,z) + h^{t,n}(z,y)
	\end{align*}
	Taking the liminf on $n$, we get 
	\begin{equation*}
		h^{s,\infty}(x,y) \leq h^{s,t}(x,z) +  h^{t,\infty}(z,y)
	\end{equation*}
	which holds for all points $z$ of $M$. Changing the sign and taking the supremum on $z$ gives the desired inequality
	\begin{equation*}
		v(s,x) = - h^{s,\infty}(x,y)  \geq \mathcal{T}_{+,0}^{t,s} v(t)(x) = \sup_{z\in M} \big\{ - h^{t,\infty}(z,y) - h^{s,t}(x,z) \big\}
	\end{equation*} \\
		
	We have shown that $\mathcal{T}_{+,0}^{t,s} v(t, \cdot) = v(s, \cdot)$. Let us establish the desired properties on the map $u(t,x)$. We have
	\begin{align*}
		\mathcal{T}_+^{t,s}u(t, \cdot) &= \mathcal{T}_{+,0}^{t,s}u(t, \cdot) - \alpha_0.(t-s) = \mathcal{T}_{+,0}^{t,s} \big( u(t, \cdot) - \alpha_0.t \big) +\alpha_0.s \\
		&= \mathcal{T}_{+,0}^{t,s} v(t, \cdot) +\alpha_0.s = v(s, \cdot) + \alpha_0.s  = u(s, \cdot)
	\end{align*}
	Additionally, we have
	\begin{align*}
		v(t+1,x) &= -h^{t+1,\infty}(x,y) = - \liminf_n h^{t+1,n}(x,y) = - \liminf_n h^{t,n-1}(x,y) = -h^{t,\infty}(x,y) = v(t,x) \\
	\end{align*}
	where we used the time periodicity of the Lagrangian $L$. This yields
	\begin{align*}
		\mathcal{T}_+^{0,-1} u(0, \cdot) = \mathcal{T}_{+,0}^{0,-1}u(t, \cdot) - \alpha_0 = u(-1, \cdot) - \alpha_0 = v(-1, \cdot) = v(0, \cdot) = u(0,\cdot)
	\end{align*}
\end{proof}

\begin{prop} \label{BoundedLO}
	For any time $t \in \mathbb{R}$ and any scalar map $u \in \mathcal{C}^0(M, \mathbb{R})$, the family of maps $\big( \mathcal{T}_+^{t,s}u = \mathcal{T}_{+,0}^{t,s}u -\alpha_0.(t-s) \big)_{s<t}$ is uniformly bounded in the $C^0$ topology.
\end{prop}

\begin{proof}
	This is mainly due to the non-expansiveness of the Lax-Oleinik operator $\mathcal{T}_+^{t,s}$ and to the existence of positive weak-KAM solutions. Fix a time $t \in \mathbb{R}$ and a scalar map $u \in \mathcal{C}^0(M, \mathbb{R})$ and let $w(\tau,x) = - h^{\tau,\infty}(x,y) + \alpha_0.\tau$ be a positive weak solution with initial data $w_0 = w(0, \cdot)$, introduced in Proposition \ref{WeakKAM+}. The Proposition \ref{NonExpansiveness} yields for all time $s \leq t$, 
	\begin{align*}
		\Vert  \mathcal{T}_+^{t,s}u - \mathcal{T}_+^{t,s} \mathcal{T}_+^{ \lceil t \rceil , t}w_0 \Vert _\infty \leq \Vert  u-\mathcal{T}_+^{ \lceil t \rceil , t}w_0 \Vert _\infty
	\end{align*}
	where $\lceil \cdot \rceil$ stand for the ceil map. Besides, we know from the definition of positive weak-KAM solutions that
	\begin{align*}
		\mathcal{T}_+^{t,s} \mathcal{T}_+^{ \lceil t \rceil , t}w_0 =  \mathcal{T}_+^{ \lceil t \rceil , s}w_0 =  \mathcal{T}_+^{ \lceil s \rceil , s} \mathcal{T}_+^{ \lceil t \rceil , \lceil s \rceil} w_0 = \mathcal{T}_+^{ \lceil s \rceil , s} w_0 = \mathcal{T}_+^{0, s- \lceil s \rceil} w_0 
	\end{align*}
	Therefore, we obtain a uniform bound using the continuity in time $\tau$ of $\mathcal{T}_+^{\tau} w_0$ (see \cite{Representation}) as follows
	\begin{align*}
		\Vert  \mathcal{T}_+^{t,s}u \Vert _\infty & \leq \Vert  u-\mathcal{T}_+^{ \lceil t \rceil , t}w_0 \Vert _\infty + \Vert \mathcal{T}_+^{t,s} \mathcal{T}_+^{ \lceil t \rceil , t}w_0\Vert _\infty \\
		&= \Vert u-\mathcal{T}_+^{ \lceil t \rceil , t}w_0  \Vert _\infty + \Vert \mathcal{T}_+^{0, s- \lceil s \rceil} w_0  \Vert _\infty \\
		& \leq \Vert  u-\mathcal{T}_+^{ \lceil t \rceil , t}w_0 \Vert _\infty + \sup_{\tau \in [-1,0]} \Vert \mathcal{T}_+^{0,\tau}w_0 \Vert \Vert _\infty < +\infty
	\end{align*}
\end{proof}

\paragraph{Acknowlegements :} The author is sincerely grateful to Marie-Claude Arnaud for her thorough reading, and to Jacques Féjoz for the many enlightening discussions that greatly contributed to this work.

\bibliographystyle{alpha}
\addcontentsline{toc}{section}{References}
\bibliography{Biblio}

\end{document}